\documentclass[12pt]{article}
\usepackage[english]{babel}
\usepackage[top = 1in, bottom = 1in, left = 1in, right = 1in]{geometry}
\usepackage{amsmath,amsxtra,amssymb,latexsym,amscd,amsthm,enumitem,calligra,mathrsfs}
\usepackage{graphicx}
\usepackage[mathscr]{eucal}
\usepackage{fancybox}
\usepackage{xparse}
\usepackage{environ}
\usepackage[all]{xy}
\usepackage{tikz}
\usepackage{tikz-cd}
\usepackage{amscd}
\usepackage{mathtools}
\usepackage{stmaryrd}
\usepackage[backend=biber,style=alphabetic,]{biblatex}
\usetikzlibrary{positioning}
\usetikzlibrary{arrows}
\usepackage{hyperref}
\usepackage[nottoc]{tocbibind}
\usepackage{imakeidx}
\makeindex[columns=2, title=Index, intoc]

\ExplSyntaxOn
\NewEnviron{pmatrixT}
 {
  \marine_transpose:V \BODY
 }

\int_new:N \l_marine_transpose_row_int
\int_new:N \l_marine_transpose_col_int
\seq_new:N \l_marine_transpose_rows_seq
\seq_new:N \l_marine_transpose_arow_seq
\prop_new:N \l_marine_transpose_matrix_prop
\tl_new:N \l_marine_transpose_last_tl
\tl_new:N \l_marine_transpose_body_tl

\cs_new_protected:Nn \marine_transpose:n
 {
  \seq_set_split:Nnn \l_marine_transpose_rows_seq { \\ } { #1 }
  \int_zero:N \l_marine_transpose_row_int
  \prop_clear:N \l_marine_transpose_matrix_prop
  \seq_map_inline:Nn \l_marine_transpose_rows_seq
   {
    \int_incr:N \l_marine_transpose_row_int
    \int_zero:N \l_marine_transpose_col_int
    \seq_set_split:Nnn \l_marine_transpose_arow_seq { & } { ##1 }
    \seq_map_inline:Nn \l_marine_transpose_arow_seq
     {
      \int_incr:N \l_marine_transpose_col_int
      \prop_put:Nxn \l_marine_transpose_matrix_prop
       {
        \int_to_arabic:n { \l_marine_transpose_row_int }
        ,
        \int_to_arabic:n { \l_marine_transpose_col_int }
       }
       { ####1 }
     }
   }
   \tl_clear:N \l_marine_transpose_body_tl
   \int_step_inline:nnnn { 1 } { 1 } { \l_marine_transpose_col_int }
    {
     \int_step_inline:nnnn { 1 } { 1 } { \l_marine_transpose_row_int }
      {
       \tl_put_right:Nx \l_marine_transpose_body_tl
        {
         \prop_item:Nn \l_marine_transpose_matrix_prop { ####1,##1 }
         \int_compare:nF { ####1 = \l_marine_transpose_row_int } { & }
        }
      }
     \tl_put_right:Nn \l_marine_transpose_body_tl { \\ }
    }
   \begin{pmatrix}
   \l_marine_transpose_body_tl
   \end{pmatrix}
 }
\cs_generate_variant:Nn \marine_transpose:n { V }
\cs_generate_variant:Nn \prop_put:Nnn { Nx }
\ExplSyntaxOff

\DeclareFontFamily{U}{wncy}{}
\DeclareFontShape{U}{wncy}{m}{n}{<->wncyr10}{}
\DeclareSymbolFont{mcy}{U}{wncy}{m}{n}
\DeclareMathSymbol{\Sh}{\mathord}{mcy}{"58}

\newcommand{\Rbb}{\mathbb{R}}

\newcommand{\Fbb}{\mathbb{F}}
\newcommand{\Nbb}{\mathbb{N}}
\newcommand{\Qbb}{\mathbb{Q}}
\newcommand{\Zbb}{\mathbb{Z}}
\newcommand{\Fpbb}{\mathbb{F}_{p}}

\newcommand{\Qpbb}{\mathbb{Q}_{p}}
\newcommand{\Zpbb}{\mathbb{Z}_{p}}
\newcommand{\Pbb}{\mathbb{P}}
\newcommand{\Abb}{\mathbb{A}}
\newcommand{\Gmbb}{\mathbb{G}_{m}}

\newcommand{\ds}{\displaystyle}

\theoremstyle{definition}
\newtheorem{definition}{Definition}[section]
\newtheorem{remark}[definition]{Remark}
\newtheorem{example}[definition]{Example}

\theoremstyle{plain}
\newtheorem{theorem}{Theorem}[section]
\newtheorem{proposition}[theorem]{Proposition}

\newtheorem{corollary}[theorem]{Corollary}

\begin{document}

%\begin{titlepage}

%\center

%\begin{minipage}{1 \textwidth}
	%\title{\textbf{Brauer--Manin obstruction for integral points on Markoff-type cubic surfaces}}
	%\author{\textsc{Quang-Duc DAO} \\ Advisor : Prof. \textsc{Cyril DEMARCHE}}
	%\date{IMJ-PRG}
	%\maketitle
%\end{minipage}\\[5cm]
\title{\textbf{Rational and integral points on Markoff-type K3 surfaces}}
	\author{\textsc{Quang-Duc Dao}} %\\ Advisor : Prof. \textsc{Cyril DEMARCHE}}
	\date{}
	\maketitle
%\begin{center}
	%\includegraphics[]{Logo_IMJ-PRG.png}
	%[width=4in,angle=0]
	%\hspace{5cm}	
	%\includegraphics[width=1.5in,angle=0]{Logo_IMH.png}
	%\hspace{5cm}	
	%\includegraphics[width=1.5in,angle=0]{flag_yellow_low.jpg}
%\end{center}
%\vfill
%\end{titlepage}

\begin{abstract}
Following recent works by E. Fuchs \textit{et al.} and by the author, we study rational and integral points on Markoff-type K3 (MK3) surfaces, i.e., Wehler K3 surfaces of Markoff type. In particular, we construct a family of MK3 surfaces which have a Zariski-dense set of rational points but fail the integral Hasse principle due to the Brauer--Manin obstruction and provide some counting results for this family. We also give some remarks on Brauer groups, Picard groups, and failure of strong approximation on MK3 surfaces.
\end{abstract}

\section{Introduction}
Let $X$ be an affine variety over $\Qbb$, and $\mathcal{X}$ an integral model of $X$ over $\Zbb$, which is an affine scheme of finite type over $\Zbb$ whose generic fiber is isomorphic to $X$. Define the set of adelic points $X(\textbf{\textup{A}}_{\Qbb}) := \sideset{}{'}\prod_{p} X(\Qpbb)$, where $p$ is a prime number or $p = \infty$ (with $\Qbb_{\infty} = \Rbb$). Similarly, define $\mathcal{X}(\textbf{\textup{A}}_{\Zbb}) := \prod_{p} \mathcal{X}(\Zpbb)$ (with $\Zbb_{\infty} = \Rbb$). We say that $X$ \textit{fails the Hasse principle} if 
$$ X(\textbf{\textup{A}}_{\Qbb}) \not= \emptyset \hspace{0.5cm}\textup{but}\hspace{0.5cm} X(\Qbb) = \emptyset $$
We say that $\mathcal{X}$ \textit{fails the integral Hasse principle} if 
$$ \mathcal{X}(\textbf{\textup{A}}_{\Zbb}) \not= \emptyset \hspace{0.5cm}\textup{but}\hspace{0.5cm} \mathcal{X}(\Zbb) = \emptyset.) $$
We say that $X$ \textit{satisfies weak approximation} if the image of $X(\Qbb)$ in $\prod_{v} X(\Qbb_{v})$ is dense, where the product is taken over all places of $\Qbb$. Finally, we say that $\mathcal{X}$ \textit{satisfies strong approximation} if $\mathcal{X}(\Zbb)$ is dense in $\mathcal{X}(\textbf{A}_{\Zbb})_{\bullet} := \prod_{p} \mathcal{X}(\Zpbb) \times \pi_{0}(X(\Rbb))$, where $\pi_{0}(X(\Rbb))$ denotes the set of connected components of $X(\Rbb)$. %Note that we work with $\pi_{0}(X(\Rbb))$ since $\mathcal{X}(\Zbb)$ is never dense in $X(\Rbb)$ for simple topological reasons. 

In general, few varieties satisfy the Hasse principle. In his 1970 ICM address \cite{Man71}, Manin introduced a natural cohomological obstruction to the Hasse principle, namely the \textbf{Brauer--Manin obstruction} (which has been extended to its integral version in \cite{CTX09}). If $\textup{Br}\,X$ denotes the cohomological Brauer group of $X$, i.e., $\textup{Br}\,X := \textup{H}^{2}_{\textup{ét}}(X,\Gmbb)$, we have a natural pairing from class field theory:
$$ X(\textbf{\textup{A}}_{\Qbb}) \times \textup{Br}\,X \rightarrow \Qbb/\Zbb. $$
If we define $X(\textbf{\textup{A}}_{\Qbb})^{\textup{Br}}$ to be the left kernel of this pairing, then the exact sequence of Albert--Brauer--Hasse--Noether gives us the relation:
$$ X(\Qbb) \subseteq X(\textbf{\textup{A}}_{\Qbb})^{\textup{Br}} \subseteq X(\textbf{\textup{A}}_{\Qbb}). $$
Similarly, by defining the Brauer--Manin set $\mathcal{X}(\textbf{A}_{\Zbb})^{\text{Br}}$, we also have that 
$$ \mathcal{X}(\Zbb) \subseteq \mathcal{X}(\textbf{A}_{\Zbb})^{\text{Br}} \subseteq \mathcal{X}(\textbf{A}_{\Zbb}). $$
This gives the so-called \textit{integral} Brauer--Manin obstruction. We say that \textit{the Brauer--Manin obstruction to the (resp. integral) Hasse principle is the only one} if 

$$ X(\textbf{A}_{\Qbb})^{\text{Br}} \not= \emptyset \iff X(\Qbb) \not= \emptyset.$$

$$ (\mathcal{X}(\textbf{A}_{\Zbb})^{\text{Br}} \not= \emptyset \iff \mathcal{X}(\Zbb) \not= \emptyset.)$$

We are interested particularly in the case where $X$ is a hypersurface, defined by a polynomial equation of degree $d$ in an affine space. The case $d = 1$ is easy and elementary. The case $d = 2$ considers the arithmetic of quadratic forms: for rational points, the Hasse principle is always satisfied by the Hasse--Minkowski theorem, and for integral points, the Brauer--Manin obstruction to the integral Hasse principle is the only one (up to an isotropy assumption) due to work of Colliot-Thélène, Xu \cite{CTX09} and Harari \cite{Ha08}.

However, the case $d = 3$ (of cubic hypersurfaces) is still largely open, especially for integral points. Overall, the arithmetic of integral points on the affine cubic surfaces over number fields is still little understood. An interesting example of affine cubic surfaces is given by \textbf{Markoff surfaces} $U_{m}$ which are defined by 
$$ x^{2} + y^{2} + z^{2} - xyz = m, $$ 
where $m$ is an integer parameter. In \cite{GS22}, Ghosh and Sarnak study the integral points on those affine Markoff surfaces $U_{m}$ both from a theoretical point of view and from numerical evidence. They prove that for almost all $m$, the integral Hasse principle holds, and that there are infinitely many $m$'s for which it fails (Hasse failures). Furthermore, their numerical experiments suggest particularly a proportion of integers $m$ satisfying $|m| \leq M$ of the power $M^{0,8875\dots+o(1)}$ for which the integral Hasse principle is not satisfied.

Subsequently, Loughran and Mitankin \cite{LM20} proved that asymptotically only a proportion of $M^{1/2}/(\log M)^{1/2}$ of integers $m$ such that $-M \leq m \leq M$ presents an integral Brauer--Manin obstruction to the Hasse principle. They also obtained a lower bound, asymptotically $M^{1/2}/\log M$, for the number of Hasse failures which cannot be explained by the Brauer--Manin obstruction. After Colliot-Thélène, Wei, and Xu \cite{CTWX20} obtained a slightly stronger lower bound than the one given in \cite{LM20}, no better result than their number $M^{1/2}/(\log M)^{1/2}$ has been known until now.
\\~\\
\indent In recent work \cite{Dao24a}, we study the set of integral points of a different Markoff-type cubic surfaces whose origin is similar to that of the original Markoff surfaces $U_{m}$, namely the \textit{relative character varieties}, using the Brauer--Manin obstruction as well. The surfaces are given by the cubic equation:

$$ x^{2} + y^{2} + z^{2} + xyz = ax + by + cz + d, $$

\noindent where $a,b,c,d \in \Zbb$ are parameters which satisfy some specific relations (see \cite{CL09}). Due to the similar appearance to the original Markoff surfaces, one may expect to find some similarities in their arithmetic properties. One of the main results in \cite{Dao24} is that a positive proportion of these relative character varieties have no (algebraic) Brauer--Manin obstruction to the integral Hasse principle as well as fail strong approximation, which can be explained by the Brauer--Manin obstruction.

Finally, in this paper, following recent work \cite{Dao24b}, we continue to study the arithmetic of Markoff-type K3 surfaces. Let $K$ be a number field. Denote a point in $\Pbb^{1} \times \Pbb^{1} \times \Pbb^{1}$ by $([x:r], [y:s], [z:t])$. Let $X \subset \Pbb^{1} \times \Pbb^{1} \times \Pbb^{1}$ be a (not necessarily smooth) surface over $K$, given by a $(2, 2, 2)$ form
$$ F(x, r ; y, s ; z, t) \in K[x, r ; y, s ; z, t]. $$
Then $X$ is called a \emph{Wehler surface}. If $X$ is smooth, $X$ is an elliptic K3 surface whose projections $p_{i} : X \rightarrow \Pbb^{1}$ ($i \in \{1,2,3\}$) have fibers as curves of (arithmetic) genus $1$.

A \textbf{Markoff-type K3 surface} $W$ is a Wehler surface whose $(2, 2, 2)$-form $F$ is invariant under the action of the group $\mathcal{G} \subset \text{Aut}(\Pbb^{1} \times \Pbb^{1} \times \Pbb^{1})$ generated by $(x,y,z) \mapsto (-x,-y,z)$ and permutations of $(x,y,z)$. By \cite{FLST23}, there exist $a, b, c, d, e \in K$ so that the $(2, 2, 2)$-form $F$ that defines $W$ has the affine form:
$$ ax^{2}y^{2}z^{2} + b(x^{2}y^{2} + y^{2}z^{2} + z^{2}x^{2}) + cxyz + d(x^{2} + y^{2} + z^{2}) + e = 0. $$

Our main result shows the Brauer--Manin obstruction with respect to explicit elements of the algebraic Brauer group for the existence of integral points on a concrete family of Markoff-type K3 (MK3) surfaces.

\begin{theorem}
For $k \in \Zbb$, let $W_{k} \subset \Pbb^{1} \times \Pbb^{1} \times \Pbb^{1}$ be the MK3 surface defined over $\Qbb$ by the $(2,2,2)$-form 
\begin{equation}
    F(x,y,z) = (x^{2} - 36)(y^{2} - 36)(z^{2} - 36) - m_{0}(xyz + C_{0})^{2} - k = 0.
\end{equation}
Let $\mathcal{U}_{k}$ be the integral model of $U_{k}$ defined over $\Zbb$ by the same equation. If $k$ satisfies the condition:
\begin{enumerate}
    \item $k = \ell^{2}$ where $\ell \in \Zbb, \ell > 1$ such that $\ell \equiv 1$ \textup{mod} $P$, where $P = 2^{3} \times 3 \times 5 \times 7 \times 11 \times 13 \times 31 \times 433 \times 2017 \times 3253 \times 8501 \times 32687 \times 46649 \times 4057231$;
    \item $(p,13)_{p} = 0$ for any prime divisor $p$ of $\ell$,
\end{enumerate}
then $\mathcal{U}_{k}(\textbf{\textup{A}}_{\Zbb}) \not= \emptyset$ and there is an algebraic Brauer--Manin obstruction to the integral Hasse principle for $\mathcal{U}_{k}$, i.e., $\mathcal{U}_{k}(\Zbb) \subset \mathcal{U}_{k}(\textbf{\textup{A}}_{\Zbb})^{\textup{Br}_{1}} = \emptyset$.
\end{theorem}

We also have a counting result on the number of counterexamples to the integral Hasse principle for MK3 surfaces. This asymptotic result is the same as that in \cite{Dao24b}. Recall that for Markoff surfaces, Loughran and Mitankin \cite{LM20} proved that asymptotically a proportion of $M^{1/2}/(\log M)^{1/2}$ of integers $m$ such that $|m| \leq M$ presents an integral Brauer--Manin obstruction.

\begin{theorem}
For the above family of MK3 surfaces, we have
	$$ \# \{k \in \Zbb: |k| \leq M,\ \mathcal{U}_{k}(\textbf{\textup{A}}_{\Zbb}) \not= \emptyset \} \asymp M $$
	and
	$$ \# \{k \in \Zbb: |k| \leq M,\ \mathcal{U}_{k}(\textbf{\textup{A}}_{\Zbb}) \not= \emptyset,\ \mathcal{U}_{k}(\textbf{\textup{A}}_{\Zbb})^{\textup{Br}} = \emptyset \} \gg \frac{M^{1/2}}{\textup{log}\,M}, $$
as $M \rightarrow +\infty$.
\end{theorem}

The structure of the paper is as follows. In Section 2, we provide some background on Wehler K3 surfaces and a recent study of the Markoff-type K3 (MK3) surfaces. In Section 3, we study rational points on a general family of MK3 surfaces and prove that they are Zariski-dense. In Section 4, we first discuss some geometry of Wehler K3 surfaces and their Brauer groups. After the general setting, we turn our attention to a particular family of MK3 surfaces, where we explicitly compute the geometric Picard group and the algebraic Brauer group of the projective closures, and then we compute the algebraic Brauer group of the affine surfaces. In Section 5, we use the algebraic Brauer group to give explicit examples of Brauer--Manin obstructions to the integral Hasse principle for the family of MK3 surfaces, and give some counting results for the Hasse failures. Finally, in Section 6, we give some remarks about Brauer groups, Picard groups, and failure of strong approximation for MK3 surfaces.
\\~\\
\indent \textbf{Notation.} Let $k$ be a field and $\overline{k}$ a separable closure of $k$. We let $G_{k} := \textup{Gal}(\overline{k}/k)$ be the absolute Galois group. A $k$-variety is a separated $k$-scheme of finite type. If $X$ is a $k$-variety, we write $\overline{X} = X \times_{k} \overline{k}$. Let $k[X] = \textup{H}^{0}(X,\mathcal{O}_{X})$ and $\overline{k}[X] = \textup{H}^{0}(\overline{X},\mathcal{O}_{\overline{X}})$. If $X$ is an integral $k$-variety, let $k(X)$ denote the function field of $X$. If $X$ is a geometrically integral $k$-variety, let $\overline{k}(X)$ denote the function field of $\overline{X}$. 

Let $\textup{Pic}\,X = \textup{H}^{1}_{\textup{Zar}}(X,\Gmbb) = \textup{H}^{1}_{\textup{ét}}(X,\Gmbb)$ denote the Picard group of a scheme $X$. Let $\textup{Br}\,X = \textup{H}^{2}_{\textup{ét}}(X,\Gmbb)$ denote the Brauer group of $X$. Let
$$ \textup{Br}_{1}\,X := \textup{Ker}[\textup{Br}\,X \rightarrow \textup{Br}\,\overline{X}] $$
denote the \textbf{algebraic Brauer group} of a $k$-variety $X$ and let $\textup{Br}_{0}\,X \subset \textup{Br}_{1}\,X$ denote the image of $\textup{Br}\,k \rightarrow \textup{Br}\,X$. The image of $\textup{Br}\,X \rightarrow \textup{Br}\,\overline{X}$ is called the \textbf{transcendental Brauer group} of $X$.

Given a field $F$ of characteristic zero containing a primitive $n$-th root of unity $\zeta = \zeta_{n}$, we have $\textup{H}^{2}(F,\mu_{n}^{\otimes 2}) = \textup{H}^{2}(F,\mu_{n}) \otimes \mu_{n}$. The choice of $\zeta_{n}$ then defines an isomorphism $\textup{Br}(F)[n] = \textup{H}^{2}(F,\mu_{n}) \cong \textup{H}^{2}(F,\mu_{n}^{\otimes 2})$. Given two elements $f, g \in F^{\times}$, we have their classes $(f)$ and $(g)$ in $F^{\times}/F^{\times n} = \textup{H}^{1}(F,\mu_{n})$. We denote by $(f,g)_{\zeta} \in \textup{Br}(F)[n] = \textup{H}^{2}(F,\mu_{n})$ the class corresponding to the cup-product $(f) \cup (g) \in \textup{H}^{2}(F,\mu_{n}^{\otimes 2})$. Suppose $F/E$ is a finite Galois extension with Galois group $G$. Given $\sigma \in G$ and $f,g \in F^{\times}$, we have $\sigma((f,g)_{\zeta_{n}}) = (\sigma(f),\sigma(g))_{\sigma(\zeta_{n})} \in \textup{Br}(F)$. In particular, if $\zeta_{n} \in E$, then $\sigma((f,g)_{\zeta_{n}}) = (\sigma(f),\sigma(g))_{\zeta_{n}}$. For all the details, see \cite[Section 4.6, Section 4.7]{GS17}.

Let $R$ be a discrete valuation ring with fraction field $F$ and residue field $\kappa$. Let $v$ denote the valuation $F^{\times} \rightarrow \Zbb$. Let $n > 1$ be an integer invertible in $R$. Assume that $F$ contains a primitive $n$-th root of unity $\zeta$. For $f,g \in F^{\times}$, we have the residue map
$$ \partial_{R} : \textup{H}^{2}(F,\mu_{n}) \rightarrow \textup{H}^{1}(\kappa,\Zbb/n\Zbb) \cong \textup{H}^{1}(\kappa,\mu_{n}) = \kappa^{\times}/\kappa^{\times n}, $$
where $\textup{H}^{1}(\kappa,\Zbb/n\Zbb) \cong \textup{H}^{1}(\kappa,\mu_{n})$ is induced by the isomorphism $\Zbb/n\Zbb \simeq \mu_{n}$ sending $1$ to $\zeta$. This map sends the class of $(f,g)_{\zeta} \in \textup{Br}(F)[n] = \textup{H}^{2}(F,\mu_{n})$ to 
$$ (-1)^{v(f)v(g)} \textup{class}(g^{v(f)}/f^{v(g)}) \in \kappa/\kappa^{\times n}. $$

For a proof of these facts, see \cite{GS17}. Here we recall some precise references. Residues in Galois cohomology with finite coefficients are defined in \cite[Construction 6.8.5]{GS17}. Comparison of residues in Milnor K-Theory and Galois cohomology is given in \cite[Proposition 7.5.1]{GS17}. The explicit formula for the residue in Milnor’s group K2 of a discretely valued field is given in \cite[Example 7.1.5]{GS17}.
\\~\\
\indent \textbf{Acknowledgements.} I would like to express my sincere gratitude to Cyril Demarche for his excellent supervision during my PhD study at Sorbonne University in Paris, France; this research is a continuation from my PhD thesis. I deeply thank Pho Duc Tai, my former undergraduate thesis supervisor at VNU Hanoi University of Science, for his valuable help with the computation of numbers of rational points on Wehler K3 surfaces over finite fields on SageMath. I am sincerely grateful to Joseph H Silverman for his comments on this paper and some information on his article with Fuchs et al. \cite{FLST23}.

\section{Background}
We give some notation and results on Wehler K3 surfaces and the so-called \emph{Markoff-type} K3 surfaces that we study in this paper.

\subsection{Wehler K3 surfaces}
Consider the variety $M = \Pbb^{1} \times \Pbb^{1} \times \Pbb^{1}$ and let $\pi_{1}$, $\pi_{2}$, and $\pi_{3}$ be the projections on the first, second, and third factor: $\pi_{i}(z_{1},z_{2},z_{3}) = z_{i}$. Denote by $L_{i}$ the line bundle $\pi_{i}^{*}(\mathcal{O}(1))$ and set
$$ L = L_{1}^{2} \otimes L_{2}^{2} \otimes L_{3}^{2} = \pi_{1}^{*}(\mathcal{O}(2)) \otimes \pi_{2}^{*}(\mathcal{O}(2)) \otimes \pi_{3}^{*}(\mathcal{O}(2)). $$
Since $K_{\Pbb^{1}} = \mathcal{O}(-2)$, this line bundle $L$ is the dual of the canonical bundle $K_{M}$. By definition, $|L| \simeq \Pbb(\textup{H}^{0}(M,L))$ is the linear system of surfaces $W \subset M$ given by the zeroes of global sections $P \in \textup{H}^{0}(M,L)$. Using affine coordinates $(x_{1}, x_{2}, x_{3})$ on $M = \Pbb^{1} \times \Pbb^{1} \times \Pbb^{1}$, such a surface is defined by a polynomial equation $F(x_{1},x_{2},x_{3}) = 0$ whose degree with respect to each variable is $\leq 2$. These surfaces will be referred to as \textbf{Wehler surfaces}; modulo $\textup{Aut}(M)$, they form a family of dimension $17$.

Fix $k \in \{1,2,3\}$ and denote by $i < j$ the other indices. If we project $W$ to $\Pbb^{1} \times \Pbb^{1}$ by $\pi_{ij} = (\pi_{i}, \pi_{j})$, we get a $2$ to $1$ cover (the generic fiber is made of two points, but some fibers may be rational curves). As soon as $W$ is \emph{smooth}, the involution $\sigma_{k}$ that permutes the two points in each (general) fiber of $\pi_{ij}$ is an involutive automorphism of $W$; indeed $W$ is a K3 surface and any birational self-map of such a surface is an automorphism (see \cite[Lemma 1.2]{Bi97}). By \cite[Proposition 3.1]{CD23}, we have the following general result.

\begin{proposition} 
There is a countable union of proper Zariski closed subsets $(S_{i})_{i \geq 0}$ in $|L|$ such that:
\begin{enumerate}
\item[(1)] If $W$ is an element of $|L| \setminus S_{0}$, then $W$ is a smooth K3 surface and $W$ does not contain any fiber of the projections $\pi_{ij}$, i.e., each of the three projections $(\pi_{ij})_{|W} : W \rightarrow \Pbb^{1} \times \Pbb^{1}$ is a finite map;

\item[(2)] If $W$ is an element of $|L| \setminus (\cup_{i \geq 0} S_{i})$, the restriction morphism $\textup{Pic}\,M \rightarrow \textup{Pic}\,W$ is surjective. In particular, the Picard number of $W$ is equal to $3$.
\end{enumerate}
\end{proposition}

From the second assertion, we deduce that for a \emph{very general} $W$, $\textup{Pic}\,W$ is isomorphic to $\textup{Pic}\,M$: it is the free Abelian group of rank $3$, generated by the classes
$$ D_{i} := [(L_{i})_{|W}]. $$
The elements of $|(L_{i})_{|W}|$ are the curves of $W$ given by the equations $z_{i} = \alpha$ for some $\alpha \in \Pbb^{1}$. The arithmetic genus of these curves is equal to $1$: in other words, the projection $(\pi_{i})_{|W} : W \rightarrow \Pbb^{1}$ is a genus $1$ fibration (see \cite[Lemma 1.1]{Bi97}). Moreover, for a general choice of $W$ in $|L|$, $(\pi_{i})_{|W}$ has 24 singular fibers of type $\textup{I}_{1}$, which are isomorphic to a rational curve with exactly one simple double point. The intersection form is given by $D_{i}^{2} = 0$ and $(D_{i}.D_{j}) = 2$ if $i \not= j$, so that its matrix is given by
\[
\begin{pmatrix}
0 & 2 & 2\\
2 & 0 & 2\\
2 & 2 & 0
\end{pmatrix}.
\]

Note that if $W$ is a smooth K3 surface, then $\textup{Pic}\,W \simeq \textup{NS}\,W$. We have the following result (see \cite[Proposition 1.3]{Bi97} and \cite[Proposition 3.3]{CD23}):

\begin{proposition}
If $W$ is a very general Wehler surface then:
\begin{enumerate}
\item[(1)] $W$ is a smooth K3 surface with Picard number $3$;

\item[(2)] $\textup{Aut}(W) = \langle \sigma_{1}, \sigma_{2}, \sigma_{3} \rangle$, which is a free product of three copies of $\Zbb/2\Zbb$ and $\textup{Aut}(W)^{*}$ is a finite index subgroup in the group of integral isometries of $\textup{Pic}\,W = \textup{NS}\,W$. Here $\textup{Aut}(W)^{*}$ denotes the image of $\textup{Aut}(W)$ in $\textup{\textsf{GL}}(\textup{H}^{2}(W,\Zbb))$ \textup{(see \cite[Section 2.1]{CD23})}.
\end{enumerate}
\end{proposition}

Besides the three involutions $\sigma_{1}, \sigma_{2}, \sigma_{3}$, depending on the symmetries of the defining polynomial $F$, the automorphism group of a Wehler surface $W$ may contain additional automorphisms. Typical examples include symmetry in $x, y, z$ that allows permutation of the coordinates, and power symmetry that allows the signs of two of $x, y, z$ to be reversed. For example, the original Markoff equation permits these extra automorphisms; and hereafter we consider analogous Markoff-type surfaces.

Note that all the above stated results are true for \emph{very general} Wehler surfaces. In \cite{Dao24b}, our examples of surfaces are very far from being general, which leads to many different results in the end. In this paper, our family of Wehler surfaces is \emph{more general} but not very general, and it gives us different geometric and arithmetic properties to study. For example, as we will see in Section 4.2, the (geometric) Picard number of each surface in this family is equal to $8$, which is less than $18$ -- a result obtained in \cite{Dao24b}.

\subsection{Markoff-type K3 surfaces}
Now let $K$ be a field. A Wehler surface $W$ over $K$ is then a surface 
$$ W = \{\overline{F} = 0\} \subset \Pbb^{1} \times \Pbb^{1} \times \Pbb^{1} $$
defined by a $(2,2,2)$-form 
$$ \overline{F}(x,r;y,s;z,t) \in K[x,r;y,s;z,t]. $$
Using the affine coordinates $(x,y,z)$, we let 
$$ F(x,y,z) = \overline{F}(x,1;y,1;z,1), $$
and then $W$ is the closure in $\Pbb^{1} \times \Pbb^{1} \times \Pbb^{1}$ of the affine surface, which by abuse of notation we also denote by 
$$ W : F(x,y,z) = 0. $$

Following \cite[Section 3]{FLST23}, we say that $W$ is \emph{non-degenerate} if it satisfies the following two conditions:
\begin{enumerate}
\item[(i)] The projection maps $\pi_{12}, \pi_{13}, \pi_{23}$ are \emph{finite}.
\item[(ii)] The generic fibers of the projection maps $\pi_{1}, \pi_{2}, \pi_{3}$ are (irreducible, geometrically connected) smooth curves, and thus the smooth fibers are curves of genus $1$, since they are $(2, 2)$ curves in $\Pbb^{1} \times \Pbb^{1}$, i.e., curves given by the vanishing of a $(2,2)$-form.
\end{enumerate}

By analogy with the classical Markoff equation, we say that $W$ is of \emph{Markoff type} (MK3) if it is symmetric in its three coordinates and invariant under double sign changes. An MK3 surface admits a group of automorphisms $\Gamma$ generated by the three involutions, coordinate permutations, and sign changes. Following the notations in \cite{FLST23}, we define:

\begin{definition}
We let $\mathfrak{S}_{3}$, the symmetric group on $3$ letters, act on $\Pbb^{1} \times \Pbb^{1} \times \Pbb^{1}$ by permuting the coordinates, and we let the group
$$ (\mu_{2}^{3})_{1} := \{(\alpha, \beta, \gamma) : \alpha, \beta, \gamma \in \mu_{2}\, \textup{and}\, \alpha\beta\gamma = 1\} $$
act on $\Pbb^{1} \times \Pbb^{1} \times \Pbb^{1}$ via sign changes,
$$ (\alpha, \beta, \gamma)(x, y, z) = (\alpha x, \beta y, \gamma z). $$
In this way, we obtain an embedding
$$ \mathcal{G} := (\mu_{2}^{3})_{1} \rtimes \mathfrak{S}_{3} \hookrightarrow \textup{Aut}(\Pbb^{1} \times \Pbb^{1} \times \Pbb^{1}). $$
\end{definition}

\begin{definition}
A \emph{Markoff-type K3} (MK3) surface $W$ is a Wehler surface whose $(2, 2, 2)$-form $F(x,y,z)$ is invariant under the action of $\mathcal{G}$, i.e., the $(2, 2, 2)$-form $F$ defining $W$ satisfies
\[
\begin{aligned}
F(x, y, z) &= F(-x, -y, z) = F(-x, y, -z) = F(x, -y, -z),\\
F(x, y, z) &= F(z, x, y) = F(y, z, x) = F(x, z, y) = F(y, x, z) = F(z, y, x).
\end{aligned}
\]
\end{definition}

By \cite[Proposition 6.5]{FLST23}, we have the following key result about the defining form of MK3 surfaces.

\begin{proposition}
Let $W/K$ be a (possibly degenerate) MK3 surface.
\begin{enumerate}
\item[(a)] There exist $a,b,c,d,e \in K$ so that the $(2, 2, 2)$-form $F$ that defines $W$ has the form
\begin{equation}
F(x, y, z) = ax^{2}y^{2}z^{2} + b(x^{2}y^{2} + x^{2}z^{2} + y^{2}z^{2}) + cxyz + d(x^{2} + y^{2} + z^{2}) + e = 0.
\end{equation}

\item[(b)] Let $F$ be as in (a). Then $W$ is a non-degenerate, i.e., the projections $\pi_{ij} : W \rightarrow \Pbb^{1} \times \Pbb^{1}$ are \emph{quasi-finite}, if and only if
$$ c \not= 0, \hspace{0.5cm} be \not= d^{2}, \hspace{0.5cm} \textup{and} \hspace{0.5cm} ad \not= b^{2}. $$
\end{enumerate}
\end{proposition}

\begin{remark}
We can recover the original Markoff equation for a surface $S_{k}$ as a special case of a form $F$ with $a=b=0, c=-1, d=1, e=-k$. More precisely, $S_{k}$ is given by the affine equation
$$ F(x,y,z) = x^{2} + y^{2} + z^{2} - xyz - k = 0. $$
However, we note that the Markoff equation is degenerate, despite the involutions being well-defined on the affine Markoff surface $S_{k}$. This occurs because the involutions are not well-defined at some of the points at infinity in the projective closure of $S_{k}$ in $\Pbb^{1} \times \Pbb^{1} \times \Pbb^{1}$; for example, the inverse image $\pi_{12}^{-1}([1:0], [1:0])$ in $X_{k}$ is a line isomorphic to $\Pbb^{1}$.
\end{remark}

\subsection{Brauer groups and the Brauer--Manin obstruction}
\subsubsection{Brauer groups}
Now let $k$ be an arbitrary field. Recall that for a variety $X$ over $k$ there is a natural filtration on the Brauer group
$$ \textup{Br}_{0}\,X \subset \textup{Br}_{1}\,X \subset \textup{Br}\,X $$ which is defined as
$$ \textup{Br}_{0}\,X = \textup{Im}[\textup{Br}\,k \rightarrow \textup{Br}\,X], \hspace{0.5cm} \textup{Br}_{1}\,X = \textup{Ker}[\textup{Br}\,X \rightarrow \textup{Br}\,\overline{X}]. $$
Let $X$ be a variety over a field $k$ such that $\overline{k}[X]^{\times} = \overline{k}^{\times}$. By Hilbert’s Theorem 90 we have $\textup{H}^{1}(k, \overline{k}^{\times}) = 0$, then by the Hochschild--Serre spectral sequence, there is a functorial exact sequence
\begin{equation}
\begin{aligned}
    0 &\longrightarrow \textup{Pic}\,X \longrightarrow \textup{Pic}\,\overline{X}^{G_{k}} \longrightarrow \textup{Br}\,k \longrightarrow \textup{Br}_{1}\,X \\
	&\longrightarrow \textup{H}^{1}(k, \textup{Pic}\,\overline{X}) \longrightarrow \textup{Ker}[\textup{H}^{3}(k, \overline{k}^{\times}) \rightarrow \textup{H}^{3}_{\textup{\'et}}(X,\Gmbb)].
\end{aligned}
\end{equation}

\begin{remark}
Let $X$ be a variety over a field $k$ such that $\overline{k}[X]^{\times} = \overline{k}^{\times}$. This assumption $\overline{k}[X]^{\times} = \overline{k}^{\times}$ holds for any proper, geometrically connected and geometrically reduced $k$-variety $X$.
\begin{enumerate} 
\item[(1)] If $X$ has a $k$-point, then each of the maps $\textup{Br}\,k \longrightarrow \textup{Br}_{1}\,X$ and $\textup{H}^{3}(k, \overline{k}^{\times}) \rightarrow \textup{H}^{3}_{\textup{\'et}} (X, \Gmbb)$ is injective. (Then $\textup{Pic}\,X \longrightarrow \textup{Pic}\,\overline{X}^{G_{k}}$ is an isomorphism.) Therefore, we have an isomorphism
$$ \textup{Br}_{1}\,X/\textup{Br}\,k \cong \textup{H}^{1}(k, \textup{Pic}\,\overline{X}). $$ 

\item[(2)] If $k$ is a number field, then $\textup{H}^{3}(k, \overline{k}^{\times}) = 0$ from class field theory. Therefore, we have an isomorphism 
$$ \textup{Br}_{1}\,X/\textup{Br}_{0}\,X \cong \textup{H}^{1}(k, \textup{Pic}\,\overline{X}). $$
\end{enumerate}
\end{remark}

We have the following result (see \cite[Theorem 5.5.1]{CTS21}).
\begin{theorem}
Let $X$ be a smooth, projective and geometrically integral variety over a field $k$. Assume that $\textup{H}^{1}(X,\mathcal{O}_{X}) = 0$ and $\textup{NS}\,\overline{X}$ is torsion-free. Then $\textup{H}^{1}(k, \textup{Pic}\,\overline{X})$ and $\textup{Br}_{1}\,X/\textup{Br}_{0}\,X$ are finite groups.
\end{theorem}

The assumption of the above theorem is always true if $X$ is a K3 surface. Furthermore, by Skorobogatov and Zarhin, we have a stronger result for the Brauer group of K3 surfaces (see \cite[Theorem 16.7.2 and Collorary 16.7.3]{CTS21}).
\begin{theorem}
Let $X$ be a K3 surface over a field $k$ finitely generated over $\Qbb$. Then $(\textup{Br}\,\overline{X})^{G_{k}}$ is finite. Moreover, the group $\textup{Br}\,X/\textup{Br}_{0}\,X$ is finite.
\end{theorem}

\subsubsection{The Brauer--Manin obstruction}
Here we briefly recall how the Brauer--Manin obstruction works in our setting, following \cite[Section 8.2]{Poo17} and \cite[Section 1]{CTX09}. For each place $v$ of $\Qbb$ there is a pairing
$$ U(\Qbb_{v}) \times \textup{Br}\,U \rightarrow \Qbb/\Zbb $$
coming from the local invariant map 
$$ \textup{inv}_{v} : \textup{Br}\,\Qbb_{v} \rightarrow \Qbb/\Zbb $$ from local class field theory (this is an isomorphism if $v$ is a prime number). This pairing is locally constant on the left by \cite[Proposition 8.2.9]{Poo17}. Any element $\alpha \in \textup{Br}\,U$ pairs trivially on $\mathcal{U}(\Zpbb)$ for almost all primes $p$, so we obtain a pairing $U(\textbf{\textup{A}}_{\Qbb}) \times \textup{Br}\,U \rightarrow \Qbb/\Zbb$. For $B \subseteq \textup{Br}\,U$, let $U(\textbf{\textup{A}}_{\Qbb})^{B}$ be the left kernel with respect to $B$, and let $U(\textbf{\textup{A}}_{\Qbb})^{\textup{Br}} = \mathcal{U}(\textbf{\textup{A}}_{\Qbb})^{\textup{Br}\,U}$. By \cite[Corollary 8.2.6]{Poo17}, we have the inclusions $U(\Qbb) \subseteq U(\textbf{\textup{A}}_{\Qbb})^{B} \subseteq \mathcal{U}(\textbf{\textup{A}}_{\Qbb})$.

For \textbf{integral points}, as the local pairings are locally constant, we obtain a well-defined pairing 
$$ \mathcal{U}(\textbf{\textup{A}}_{\Zbb})_{\bullet} \times \textup{Br}\,U \rightarrow \Qbb/\Zbb. $$
For $B \subseteq \textup{Br}\,U$, let $\mathcal{U}(\textbf{\textup{A}}_{\Zbb})_{\bullet}^{B}$ be the left kernel with respect to $B$, and let $\mathcal{U}(\textbf{\textup{A}}_{\Zbb})_{\bullet}^{\textup{Br}} = \mathcal{U}(\textbf{\textup{A}}_{\Zbb})_{\bullet}^{\textup{Br}\,U}$. By abuse of notation, from now on we write the reduced Brauer--Manin set $\mathcal{U}(\textbf{\textup{A}}_{\Zbb})_{\bullet}^{B}$ in the standard way as $\mathcal{U}(\textbf{\textup{A}}_{\Zbb})^{B}$. We have the inclusions $\mathcal{U}(\Zbb) \subseteq \mathcal{U}(\textbf{\textup{A}}_{\Zbb})^{B} \subseteq \mathcal{U}(\textbf{\textup{A}}_{\Zbb})$, so that $B$ can obstruct the \textit{integral Hasse principle} or \textit{strong approximation} on $\mathcal{U}$.

Let $V$ be dense Zariski open in $U$. As $U$ is smooth, the set $V(\Qpbb)$ is dense in $U(\Qpbb)$ for all places $p$. Moreover, $\mathcal{U}(\Zpbb)$ is open in $U(\Qpbb)$, hence $V(\Qpbb) \cap \mathcal{U}(\Zpbb)$ is dense in $\mathcal{U}(\Zpbb)$. As the local pairings are locally constant, we may restrict our attention to $V$ to calculate the local invariants of a given element in $\textup{Br}\,U$.

\section{Rational points on MK3 surfaces}
In this section, we study the existence of rational points on Markoff-type K3 surfaces. For Markoff surfaces over $\Qbb$, we know from \cite{Kol02}, \cite{LM20} and \cite{CTWX20} that there are always rational points on smooth affine Markoff surfaces; moreover, the set of rational points is Zariski-dense. However, such a phenomenon does not happen for smooth affine MK3 surfaces, since their projective closures are elliptic surfaces and lie in $(\Pbb^{1})^{3}$ instead of $\Pbb^{3}$. A natural idea to produce rational points is to construct an equation of MK3 surfaces such that there exists a fiber of $\pi_{i}$ ($i \in \{1,2,3\}$) which is an elliptic curve containing infinitely many rational points.

With this idea, we have the following result.

\begin{proposition}
For $A, B, C \in \Zbb \setminus \{0\}$ and $m \in \Zbb \setminus \{0,1\}$, let $W \subset \Pbb^{1} \times \Pbb^{1} \times \Pbb^{1}$ be the MK3 surface defined over $\Qbb$ by the $(2,2,2)$-form 
\begin{equation}
    F(x,y,z) = (x^{2}-A^{2})(y^{2}-A^{2})(z^{2}-A^{2}) - m(xyz+C)^{2} - B = 0.
\end{equation}
Let $U \subset W$ be the affine open subvariety defined by $\{rst \not= 0\}$ over $\Qbb$ by the same equation. Let $W_{0}$ be the fiber $\pi_{3}^{-1}([z:t] = [1:0])$. Then there exists $m \in \Zbb \setminus \{0,1\}$ such that there are infinitely many $\Qbb$-points on $W_{0}$ for any $A,B,C \in \Zbb \setminus \{0\}$.
\end{proposition}

\begin{proof}
For $[z:t] = [1:0]$, we obtain the affine equation for $W_{0}$:
$$ (x^{2} - A^{2})(y^{2} - A^{2}) - mx^{2}y^{2} = 0. $$
Equivalently, we have
$$ (1-m)\left(\frac{x}{A}\right)^{2}\left(\frac{y}{A}\right)^{2} + 1 = \left(\frac{x}{A}\right)^{2} + \left(\frac{y}{A}\right)^{2}. $$
This is an Edwards normal form for an elliptic curve, and from \cite{DN17}, the curve $W_{0}$ is birationally equivalent to an elliptic curve $E$ with the Weierstrass normal form
$$ \frac{v^{2}}{m} = u^{3} + \frac{4-2m}{m}u^{2} + u $$
via the rational maps 
\[
\begin{aligned}
(x,y) \mapsto (u,v) &= \left(\frac{1+y}{1-y}, \frac{2(1+y)}{x(1-y)}\right),\\
(u,v) \mapsto (x,y) &= \left(\frac{2u}{v}, \frac{u-1}{u+1}\right).
\end{aligned}
\]
The above Weierstrass form is equivalent to 
$$ (mv)^{2} = (mu)^{3} + (4-2m)(mu)^{2} + m^{2}(mu). $$ 
By letting $V = mv$ and $U = mu$, we obtain
$$ V^{2} = U^{3} + (4-2m)U^{2} + m^{2}U = U((U+2-m)^{2} - 4(1-m)). $$

Now let us consider $m = 3\ell(1-\ell)$ for any $\ell \in \Zbb \setminus \{0,1\}$. Then $E$ is a smooth elliptic curve and contains a rational point $(U,V)$ where
\[ 
\begin{aligned}
&U = \frac{16(1-m)+1}{4} - (2-m) = \frac{3(3-4m)}{4} = \frac{9}{4}(2\ell-1)^{2},\\
&V = \frac{3}{2}(2\ell - 1)\frac{16(1-m)-1}{4} = \frac{9}{8}(2\ell-1)(16\ell^{2}-16\ell+5).
\end{aligned}
\]
By the Nagell--Lutz theorem, $(U,V)$ is a rational point of infinite order. Therefore, there are infinitely many $\Qbb$-points on $E$. The same also holds for $W_{0}$ by birational equivalence.
\end{proof}

We then deduce the following result.
\begin{proposition}
With the same notation as in Proposition 3.1, assume further that $-mB$ is not a square in $\Zbb$. Then rational points on $W$ are Zariski-dense.
\end{proposition}

\begin{proof}
Under our assumptions, if we write the equation of $W$ in the form (2), we have $c = -2mC \neq 0$, $ad = (1-m)A^{4} \neq A^{4} = b^{2}$ (since $m \neq 0$ and $A \neq 0$), and $be = A^{2}(A^{6} + mC^{2} + B) \neq A^{8} = d^{2}$ (since $mC^{2} + B \neq 0$ as $-mB$ is not a square in $\Zbb$). Therefore, by Proposition 2.3, $W$ is non-degenerate, so $\pi_{ij}: W \rightarrow \Pbb^{1} \times \Pbb^{1}$ is finite, i.e., a double cover, giving rise to an involution $\sigma_{k}$. Let $\mathcal{A} = \langle \sigma_{1}, \sigma_{2}, \sigma_{3} \rangle \subset \textup{Aut}(W)$. For any point $P \in W(\overline{\Qbb})$, consider $\mathcal{C}(P) = \{\phi P \; | \; \phi \in \mathcal{A} \}$.

By \cite[Theorem 1.6]{Wan95}, the set $\{ P \in W(\Qbb) \; | \; |\mathcal{C}(P)| < +\infty \}$ is finite, so by Proposition 3.1 there is some $P \in W(\Qbb)$ such that the orbit $\mathcal{C}(P)$ is infinite. By \cite[Theorem 1.5]{Wan95}, this orbit is Zariski-dense in $W$. Therefore, $W(\Qbb)$ is Zariski-dense in $W$.
\end{proof}

\section{Brauer groups of Markoff-type K3 surfaces}
We are particularly interested in the geometry of the explicit family of Markoff-type K3 surfaces defined by $(5)$, as it is more convenient for computation and also more general than those studied in \cite{Dao24b}. In addition, under our specific conditions, the surfaces that we consider are \emph{smooth}. Before studying the arithmetic problem of integral points, we will give some explicit computations on the (geometric) Picard group and the (algebraic) Brauer group of these surfaces. Recall that by \cite[Proposition 1.3]{Bi97} or \cite[Proposition 3.3]{CD23}, for a \emph{very general} $W$, $\textup{Pic}\,\overline{W}$ is isomorphic to $\textup{Pic}(\Pbb^{1} \times \Pbb^{1} \times \Pbb^{1})$, i.e., $\textup{Pic}\,\overline{W}$ is generated by the classes $D_{i}$ so the Picard number of $\overline{W}$ equals $3$. However, we will see in this section that our example of MK3 surfaces is not very general, but not so special as those considered in \cite{Dao24b}.

\subsection{Geometry of Wehler K3 surfaces}
Let $K$ be a number field with a fixed algebraic (separable) closure $\overline{K}$. If $X$ is a K3 surface over $K$, or more generally, $X$ is a smooth, projective and geometrically integral $K$-variety such that $\textup{H}^{1}(X,\mathcal{O}_{X}) = 0$, then the Picard group $\textup{Pic}\,\overline{X}$ and the N\'eron--Severi group $\textup{NS}\,\overline{X}$ are equal (see \cite[Corollary 5.1.3]{CTS21}).

Now let $W \subset \Pbb^{1} \times \Pbb^{1} \times \Pbb^{1}$ be a smooth surface over $K$ defined by a $(2,2,2)$-form $F = 0$ (so $W$ is a Wehler K3 surface). For distinct $i,j \in \{1,2,3\}$, we keep the notations $\pi_{i} : W \rightarrow \Pbb^{1}$ and $\pi_{ij} : W \rightarrow \Pbb^{1} \times \Pbb^{1}$ of the various projections of $W$ onto one or two copies of $\Pbb^{1}$. Let $D_{i}$ denote the divisor class represented by a fiber of $\pi_{i}$. We find that $(D_{i}.D_{j}) = 2$ for $i \not= j$ and since any two different fibers of $\pi_{i}$ are disjoint, we have $D_{i}^{2} = 0$. It follows that the intersection matrix $((D_{i}.D_{j}))_{i,j}$ has rank $3$, so the $D_{i}$ generates a subgroup of rank $3$ of the N\'eron--Severi group $\textup{NS}\,\overline{W}$.

We have the following result for the geometric Picard group of Wehler surfaces.
\begin{proposition}
	Let $W \subset \Pbb^{1} \times \Pbb^{1} \times \Pbb^{1}$ be a smooth, projective, geometrically integral Wehler surface over $K$. Suppose that the three planes at infinity $\{ rst = 0 \}$ cut out on $\overline{W}$ three distinct irreducible fibers $D_{1}, D_{2}, D_{3}$ over $\overline{K}$. Let $U \subset W$ be the complement of these fibers. Then $\overline{K}^{\times} = \overline{K}[U]^{\times}$ and the natural sequence
	$$ 0 \longrightarrow \bigoplus_{i=1}^{3} \Zbb D_{i} \longrightarrow \textup{Pic}\,\overline{W} \longrightarrow \textup{Pic}\,\overline{U} \longrightarrow 0 $$
	is exact.
\end{proposition}

\begin{proof}
Note that by \cite[Proposition 1.8]{Ful98}, in order to show that the above sequence is exact, it suffices to prove that the second arrow is an injective homomorphism. Let 
$$ a_{1}D_{1} + a_{2}D_{2} + a_{3}D_{3} = 0 \in \textup{Pic}\,\overline{W} $$
with $a,b,c \in \Zbb$. By the assumption that $(D_{i}.D_{i}) = 0$ and $(D_{i}.D_{j}) = 2$ for $1 \leq i \not= j \leq 3$, one has 
$$ 2a_{2} + 2a_{3} = 2a_{1} + 2a_{3} = 2a_{1} + 2a_{2} = 0, $$
so $a_{1} = a_{2} = a_{3} = 0$. In other words, this is another proof of the fact that $D_{1}, D_{2}, D_{3}$ are linearly independent in $\textup{Pic}\,\overline{W}$ and it also shows that $\overline{K}^{\times} = \overline{K}[U]^{\times}$ as desired.
\end{proof}

Next, we will give an explicit computation of the geometric Picard group and the algebraic Brauer group for the family of Markoff-type K3 surfaces defined by $(5)$.

\subsection{The geometric Picard group}
If $X$ is a K3 surface, then linear, algebraic, and numerical equivalence all coincide (see \cite{Huy16}). This means that the Picard group $\textup{Pic}\,\overline{X}$ and the Néron--Severi group $\textup{NS}\,\overline{X}$ of $\overline{X} := X_{\overline{K}}$ are naturally isomorphic, finitely generated, and free. Their rank is called the \emph{geometric} Picard number of $X$ or the Picard number of $\overline{X}$. By the \emph{Hodge Index Theorem}, the intersection pairing on $\textup{Pic}\,\overline{X}$ is even, non-degenerate, and of signature $(1, \textup{rk}\,\textup{NS}\,\overline{X} - 1)$.

Using the explicit equations, we compute the geometric Picard group of the Markoff-type K3 surfaces in question. To bound the Picard number we use the method described in \cite{vL07b}. Let $X$ be any smooth surface over a number field $K$ and let $\mathfrak{p}$ be a prime of good reduction with residue field $\kappa$. Let $\mathcal{X}$ be an integral model for $X$ over the localization $\mathcal{O}_{\mathfrak{p}}$ of the ring of integers $\mathcal{O}$ of $K$ at $\mathfrak{p}$ for which the reduction is smooth. Let $\kappa'$ be any extension field of $\kappa$. Then by abuse of notation, we will write $X_{\kappa'}$ for $\mathcal{X} \times_{\textup{Spec}\,\mathcal{O}_{\mathfrak{p}}} \,\textup{Spec}\,\kappa'$. We need the following important result which describes the behavior of the Néron--Severi group under good reduction.

\begin{proposition}
Let $X$ be a smooth surface over a number field $K$ and let $\mathfrak{p}$ be a prime of good reduction with residue field $\kappa$. Let $l$ be a prime not dividing $q = \# \kappa$. Let $\textup{Frob}^{*}_{q}$ denote the automorphism on $\textup{H}^{2}_{\textup{ét}}(X_{\overline{\kappa}}, \Qbb_{l}(1))$ induced by the $q$-th power Frobenius. Then there are natural injective homomorphisms
$$ \textup{NS}(X_{\overline{K}}) \otimes_{\Zbb} \Qbb_{l} \hookrightarrow \textup{NS}(X_{\overline{\kappa}}) \otimes_{\Zbb} \Qbb_{l} \hookrightarrow \textup{H}^{2}_{\textup{ét}}(X_{\overline{\kappa}}, \Qbb_{l})(1) $$
of finite-dimensional vector spaces over $\Qbb_{l}$, that respect the intersection pairing and the action of Frobenius respectively. The rank of $\textup{NS}(X_{\overline{\kappa}})$ is at most the number of eigenvalues of $\textup{Frob}^{*}_{q}$ that are roots of unity, counted with multiplicity.
\end{proposition}

\begin{proof}
See \cite[Proposition 6.2 and Corollary 6.4]{vL07a} or \cite[Proposition 2.3]{BL07}.
\end{proof}

\begin{proposition}
Let $k, m, C \in \Zbb \setminus \{0\}$ and $W \subset \Pbb^{1} \times \Pbb^{1} \times \Pbb^{1}$ be a surface defined over $\Qbb$ by the $(2,2,2)$-form
$$ F(x,y,z) = (x^{2} - 36)(y^{2} - 36)(z^{2} - 36) - m(xyz + C)^{2} - k = 0. $$
If $k \equiv 1, m \equiv 1, C \equiv 3$ (\textup{mod} $7$), then $W$ is a smooth K3 surface and the Picard number of $\overline{W} = W_{\overline{\Qbb}}$ equals $8$.
\end{proposition}

\begin{proof}
Since $k \equiv 1$ mod $7$, one can check that $W_{7} := W_{\Fbb_{7}}$ is smooth, so $W$ itself is smooth and $W$ has good reduction at $p = 7$. Since $W \subset \Pbb^{1} \times \Pbb^{1} \times \Pbb^{1}$ is defined over $\Qbb$ by a $(2,2,2)$-form $F = 0$, it is a smooth elliptic K3 surface. For $i = 1, 2, 3$, let $\pi_{i} : W \rightarrow \Pbb^{1}$ be the projection from $W$ to the $i$-th copy of $\Pbb^{1}$ in $\Pbb^{1} \times \Pbb^{1} \times \Pbb^{1}$. Let $D_{i}$ denote the divisor class represented by a fiber of $\pi_{i}$. Note that since $k \equiv m \equiv 1$ mod $7$, $k,m$ are nonzero and $-km \equiv -1$ mod $7$, so $-km$ and $-k/m$ are not square in $\Qbb$. Consider the irreducible components of \emph{singular} fibers whose corresponding divisor classes on $\overline{W}$ are given explicitly as follows (denote by $[x:r], [y:s], [z:t]$ the coordinates for each point in $\Pbb^{1} \times \Pbb^{1} \times \Pbb^{1}$):

\begin{equation*}
    \begin{cases}
    C_{1}^{\pm \pm}: [x:r] = [\pm 6:1], xyz = \left(\pm \sqrt{\frac{-k}{m}} - C \right)st,\\
    C_{2}^{\pm \pm}: [y:s] = [\pm 6:1], xyz = \left(\pm \sqrt{\frac{-k}{m}} - C \right)rt,\\
    C_{3}^{\pm \pm}: [z:t] = [\pm 6:1], xyz = \left(\pm \sqrt{\frac{-k}{m}} - C \right)rs. 
    \end{cases}
\end{equation*}

We will now find explicit generators for the geometric Picard group of $W$. It is clear that $W$ is a K3 surface admitting an elliptic fibration $\pi_{1} : W \rightarrow \Pbb^{1}$ with a zero section defined by $C_{1}^{++} \simeq \Pbb^{1}$. The N\'eron--Severi group of an elliptic fibration on the K3 surface is the lattice generated by the class of a (general) fiber, the class of the zero section, the classes of the irreducible components of the reducible fibers which do not intersect the zero section, and the Mordell--Weil group (the set of the sections). Following this property, we find a set $S$ of $8$ linearly independent divisor classes in $\textup{Pic}\,\overline{W}$ consisting of:
\begin{enumerate}
	\item[(i)] $D_{1}, D_{2}, D_{3}$ (some classes of fibers);
	\item[(ii)] $C_{1}^{++}, C_{1}^{-+}, C_{2}^{++}, C_{2}^{-+}, C_{3}^{++}$ (some classes of irreducible components of singular fibers).
\end{enumerate}
Their Gram matrix of the intersection pairing on $\textup{Pic}\,\overline{W}$ has determinant $-192$, which is nonzero, so they are indeed linearly independent. More precisely, the intersection matrix associated to the sequence of divisor classes of $S$ given in the above order is 
%[ 0  2  2  0  0  1  1  1]
%[ 2  0  2  1  1  0  0  1]
%[ 2  2  0  1  1  1  1  0]
%[ 0  1  1 -2  0  1  1  1]
%[ 0  1  1  0 -2  1  1  1]
%[ 1  0  1  1  1 -2  0  1]
%[ 1  0  1  1  1  0 -2  1]
%[ 1  1  0  1  1  1  1 -2]
\[
\begin{pmatrix}
0 & 2 & 2 & 0 & 0 & 1 & 1 & 1\\
2 & 0 & 2 & 1 & 1 & 0 & 0 & 1\\
2 & 2 & 0 & 1 & 1 & 1 & 1 & 0\\
0 & 1 & 1 &-2 & 0 & 1 & 1 & 1\\
0 & 1 & 1 & 0 &-2 & 1 & 1 & 1\\
1 & 0 & 1 & 1 & 1 &-2 & 0 & 1\\
1 & 0 & 1 & 1 & 1 & 0 &-2 & 1\\
1 & 1 & 0 & 1 & 1 & 1 & 1 &-2\\
\end{pmatrix}.
\]
As a result, the Picard number of $\overline{W}$ is at least $8$.
\\~\\
\indent We will now show that the Picard number of $\overline{W}_{7}$ equals exactly $8$. Let $\overline{W}_{7}$ be the base change of $W_{7}$ to an algebraic closure of $\Fbb_{7}$, and $\textup{Frob}_{7} : \overline{W}_{7} \rightarrow \overline{W}_{7}$ the geometric Frobenius morphism, defined by $([x:r],[y:s],[z:t]) \mapsto ([x^{7}:r^{7}], [y^{7}:s^{7}], [z^{7}:t^{7}])$. Choose a prime $l \not= 7$ and let $\textup{Frob}_{7}^{*}$ be the endomorphism of $\textup{H}^{2}_{\textup{\'et}}(\overline{W}_{7}, \Qbb_{l}(1))$ induced by $\textup{Frob}_{7}$. By Proposition 3.4, the Picard rank of $\overline{W}$ is bounded above by that of $\overline{W}_{7}$, which in turn is at most the number of eigenvalues of $\textup{Frob}_{7}^{*}$ that are roots of unity. As in \cite{vL07a}, we find the characteristic polynomial of $\textup{Frob}_{7}^{*}$ by counting points on $W_{7}$. Almost all fibers of the fibration $\pi_{1}$ are smooth curves of genus 1. Using \textsc{Magma} and SageMath, we can count the number of points over small fields fiber by fiber. The first seven results are: 
$$ \begin{aligned} 
&\# W_{7}(\Fbb_{7}) = 43, \hspace{0.25cm} \# W_{7}(\Fbb_{7^{2}}) = 2843, \hspace{0.25cm} \# W_{7}(\Fbb_{7^{3}}) = 113191, \hspace{0.25cm} \# W_{7}(\Fbb_{7^{4}}) = 5786411,\\ &\# W_{7}(\Fbb_{7^{5}}) = 282458443, \hspace{0.25cm} \# W_{7}(\Fbb_{7^{6}}) = 13843757831, \hspace{0.25cm} \# W_{7}(\Fbb_{7^{7}}) = 678222249307.
\end{aligned}
$$ 

By \cite[Lemma 6.1]{vL07a}, we have $\dim \textup{H}^{i}(\overline{W}, \Qbb_{l}) = \dim \textup{H}^{i}(\overline{W}_{7}, \Qbb_{l})$ for $0 \leq i \leq 4$. Since $\overline{W}$ is a K3 surface, the Betti numbers equal $\dim \textup{H}^{i}(\overline{W}_{7}, \Qbb_{l}) = 1,0,22,0,1$ for $i = 0,1,2,3,4$, respectively. Therefore, from the Weil conjectures and the Lefschetz trace formula, we find that the trace of the $n$-th power of Frobenius acting on $\textup{H}^{2}_{\textup{\'et}}(\overline{W}_{7}, \Qbb_{l})$ equals $\# W_{7}(\Fbb_{7^{n}}) - 7^{2n} - 1$; the trace on the Tate twist $\textup{H}^{2}_{\textup{\'et}}(\overline{W}_{7}, \Qbb_{l}(1))$ is obtained by \emph{dividing} by $7^{n}$. Meanwhile, on the subspace $V \subset \textup{H}^{2}_{\textup{\'et}}(\overline{W}_{7}, \Qbb_{l}(1))$ generated over $\Qbb_{l}$ by the set $S$ of $8$ linearly independent divisor classes, $\textup{Frob}_{7}^{*}$ acts trivially on $D_{1},D_{2},D_{3}$ and $\textup{Frob}_{7}^{*}(C_{i}^{\pm \pm}) = C_{i}^{\pm \mp} = D_{i} - C_{i}^{\pm \pm}$, so the characteristic polynomial of the Frobenius acting on $V$ is $(t-1)^{3} (t+1)^{5}$. The trace $t_{n}$ is thus equal to $8$ if $n$ is even, and equal to $-2$ if $n$ is odd. Hence, on the $14$-dimensional quotient $Q = \textup{H}^{2}_{\textup{\'et}}(\overline{W}_{7}, \Qbb_{l}(1))/V$, the trace equals 
$$ \frac{\# W_{7}(\Fbb_{7^{n}})}{7^{n}} - 7^{n} - \frac{1}{7^{n}} - t_{n}. $$ 
These traces are sums of powers of eigenvalues, and we use the Newton identities to compute the elementary symmetric polynomials in these eigenvalues, which are the coefficients of the characteristic polynomial $f$ of the Frobenius acting on $Q$ (see \cite[Lemma 2.4]{vL07b}). This yields the first half of the coefficients of $f$, including the middle coefficient, which turns out to be non-zero. This implies that the sign in the functional equation $t^{14}f(1/t) = \pm f(t)$ is $+1$, so this functional equation determines $f$, which we calculate to be
$$ f(t) = t^{14} - t^{13} + 4t^{11} - 4t^{10} + 6t^{8} - 6t^{7} + 6t^{6} - 4t^{4} + 4t^{3} - t + 1. $$
As a result, we find that the characteristic polynomial of the Frobenius acting on $\textup{H}^{2}_{\textup{\'et}}(\overline{W}_{7}, \Qbb_{l}(1))$ is equal to $(t-1)^{3} (t+1)^{5} f$. The monic polynomial $f \in \Zbb[t]$ is irreducible and not cyclotomic, as there does not exist $n \in \Nbb$ such that the value of Euler's totient function at $n$ is $\varphi(n) = 14$, so its roots are not roots of unity. Thus, we obtain an upper bound of $8$ for the Picard number of $\overline{W}$.

Therefore, we deduce that $\textup{rk}\,\textup{Pic}\,\overline{W} = 8$, and the aforementioned sequence $S$ of $8$ divisor classes forms a sublattice $\Lambda \subset \textup{NS}\,\overline{W} = \textup{Pic}\,\overline{W}$ of finite index. From properties of Wehler K3 surfaces and the intersection pairings of the classes in $S$ with the other classes in the list of Proposition 3.5, we find that

\begin{equation}
\begin{cases}
D_{1} = C_{1}^{++} + C_{1}^{+-} = C_{1}^{-+} + C_{1}^{--}, \\
D_{2} = C_{2}^{++} + C_{2}^{+-} = C_{2}^{-+} + C_{2}^{--}, \\
D_{3} = C_{3}^{++} + C_{3}^{+-} = C_{3}^{-+} + C_{3}^{--}, \\
D_{1} + D_{2} + D_{3} = C_{1}^{++} + C_{1}^{-+} + C_{2}^{++} + C_{2}^{-+} + C_{3}^{++} + C_{3}^{-+}.
\end{cases}
\end{equation}

If $\Lambda$ is the whole lattice $\textup{NS}\,\overline{W}$, then $S$ forms a \emph{basis} of $\textup{Pic}\,\overline{W}$. Otherwise, assume that $\Lambda$ is a proper sublattice of $\textup{NS}\,\overline{W}$, so their discriminants differ by a square factor. We know that $\textup{disc}(\Lambda) = -192 = -2^{6} \cdot 3$, so $\Lambda$ would be a sublattice of index $2$, $4$ or $8$. In other words, there would exist a divisor class of the form 
$$ E = \frac{1}{2} \sum_{S_{i} \in S} a_{i}S_{i}, \hspace{0.5cm} a_{i} \in \{0,1\}, $$
in $\textup{Pic}\,\overline{W}$. Since the intersection pairing between $E$ and each divisor class in $S$ would give an \emph{integer} value, we find that $E$ would have to be a linear combination of the following classes:

\begin{enumerate}
	\item[(a)] $E_{1} = \frac{1}{2}(D_{1} + D_{2} + D_{3})$;
	
	\item[(b)] $E_{2} = \frac{1}{2}(C_{1}^{++} + C_{2}^{-+} + C_{3}^{++})$;
	
	\item[(c)] $E_{3} = \frac{1}{2}(C_{1}^{-+} + C_{2}^{-+} + C_{3}^{++})$;
	
	\item[(d)] $E_{4} = \frac{1}{2}(C_{2}^{++} + C_{2}^{-+})$.
\end{enumerate}

In (a), we can check that $E^{1} = 3$ is odd, which is a contradiction since the intersection pairing on $\textup{Pic}\,\overline{W}$ is \emph{even}. In (d), we have $E^{2} = -1$ is odd, which is again a contradiction. To check the divisibility by $2$ in $\textup{Pic}\,\overline{W}$, we only need to consider all the possible sums among the above classes. By considering all these sums and taking into account that $\mathcal{G} := (\mu_{2}^{3})_{1} \rtimes \mathfrak{S}_{3}$ is a subgroup of the automorphism group of $W$ (see Section 2.2) and the linear relation (deduced by using the intersection pairing)
$$ D_{1} + D_{2} + D_{3} = C_{1}^{++} + C_{1}^{-+} + C_{2}^{++} + C_{2}^{-+} + C_{3}^{++} + C_{3}^{-+},$$ 
it suffices to study the divisibility by $2$ of the two following classes:

\begin{enumerate}
	\item[(e)] $E_{5} = \frac{1}{2}(C_{1}^{++} + C_{2}^{++} + C_{3}^{++})$;
	
	\item[(f)] $E_{6} = \frac{1}{2}(D_{1} + D_{2} + D_{3} + C_{1}^{++} + C_{1}^{-+})$.
\end{enumerate}

Consider the quadratic field extension $K := \Qbb(\sqrt{-km})$ of $\Qbb$, as $-km$ is not a square in $\Qbb$. If (e) were true, by the action of the Galois group $\textup{Gal}(K/\Qbb)$ on $\textup{Pic}\,\overline{W}$, we would also have $E_{7} = \frac{1}{2}(C_{1}^{+-} + C_{2}^{+-} + C_{3}^{+-}) \in \textup{Pic}\,\overline{W}$. Then the sum of these two classes would belong to $\textup{Pic}\,\overline{W}$, so $D_{1} + D_{2} + D_{3}$ is divisible by $2$. This is a contradiction by the case (a).

In (f), we would have $E_{6}^{2} = 4$, and (e) would also imply that $E_{8} = \frac{1}{2}(D_{1} + D_{2} + D_{3} + C_{2}^{++} + C_{2}^{-+}) \in \textup{Pic}\,\overline{W}$ by the action of the symmetry group $\mathfrak{S}_{3}$ on $W$. In the set $S$, we replace $C_{1}^{-+}$ by $E_{6}$ and $C_{2}^{-+}$ by $E_{8}$ to form a new set $S'$, and then we obtain a new corresponding sublattice $\Lambda'$ of $\textup{Pic}\,\overline{W}$ whose discriminant equals $-12 = 2^{2} \cdot 3$. Then $\Lambda'$ would be the whole lattice or a proper sublattice of index $2$. We will show that the latter cannot happen. Indeed, suppose that there exists a divisor class of the form 
$$ E' = \frac{1}{2} \sum_{S'_{i} \in S'} a_{i}S'_{i}, \hspace{0.5cm} a_{i} \in \{0,1\}, $$
in $\textup{Pic}\,\overline{W}$. Since the intersection pairing between $E'$ and each divisor class in $S'$ would give an \emph{integer} value, we find that $E'$ would have to be a linear combination of the following classes:

\begin{enumerate}
	\item[(1)] $E'_{1} = \frac{1}{2}(D_{1} + D_{2} + D_{3})$;
	
	\item[(2)] $E'_{2} = \frac{1}{2}(C_{1}^{++} + C_{2}^{++} + C_{3}^{++})$.
\end{enumerate}

However, we know from above that both $E'_{1}$ and $E'_{2}$ do not belong to $\textup{Pic}\,\overline{W}$. Furthermore, $2(E'_{1} - E'_{2}) = C_{1}^{-+} + C_{2}^{-+} + C_{3}^{-+}$ is also not divisible by $2$, which can be proven by a similar argument to the one for $E'_{2}$. Therefore, we have $\Lambda' = \textup{NS}\,\overline{W} = \textup{Pic}\,\overline{W}$ with a basis given by $S'$.
\end{proof}

Next, we consider the geometric Picard group of the affine surface $U$ defined by the same equation.

\begin{corollary}
Let $U = W \setminus \{rst = 0\}$ be the affine surface defined by 
$$ F(x,y,z) = (x^{2} - 36)(y^{2} - 36)(z^{2} - 36) - m(xyz + C)^{2} - k = 0. $$
If $k \equiv 1, m \equiv 1, C \equiv 3$ (\textup{mod} $7$), then the Picard number of $\overline{U} = U_{\overline{\Qbb}}$ equals $5$.
\end{corollary}

\begin{proof}
By the exact sequence in Proposition 4.1, we obtain
$$ \textup{Pic}\,\overline{U} \cong \textup{Pic}\,\overline{W}/(\Zbb D_{1} \oplus \Zbb D_{2} \oplus \Zbb D_{3}), $$
so $\textup{Pic}\,\overline{U}$ is \emph{free} and the Picard number of $\overline{U}$ is equal to $8 - 3 = 5$, since $\{D_{1}, D_{2}, D_{3}\}$ is a part of the basis $S$ or $S'$ of $\textup{Pic}\,\overline{W}$ in the proof of Proposition 4.5.
\end{proof}

\begin{remark}
In this paper, we do not need to verify whether $S'$ can actually form a basis of $\textup{Pic}\,\overline{W}$. We will show that in both cases, for $S$ and $S'$, we obtain the same result for the algebraic Brauer groups of $W$ and $U$ in the next section.
\end{remark}

\subsection{The algebraic Brauer group}
Now given the geometric Picard group, we can compute directly the algebraic Brauer group of the Markoff-type K3 surfaces in question.

\begin{theorem}
For $k, m, C \in \Zbb$, let $W \subset \Pbb^{1} \times \Pbb^{1} \times \Pbb^{1}$ be the MK3 surface defined over $\Qbb$ by the $(2,2,2)$-form 
\begin{equation}
    F(x,y,z) = (x^{2} - 36)(y^{2} - 36)(z^{2} - 36) - m(xyz + C)^{2} - k = 0.
\end{equation}
If $k \equiv 1, m \equiv 1, C \equiv 3$ (\textup{mod} $7$) and $m \neq 1$, then
$$ \textup{Br}_{1}\,W/\textup{Br}_{0}\,W \cong (\Zbb/2\Zbb)^{2}. $$
Furthermore, for the affine surface $U = W \setminus \{rst = 0\}$, we even have 
$$ \textup{Br}_{1}\,U/\textup{Br}_{0}\,U \cong (\Zbb/2\Zbb)^{5}. $$
\end{theorem}

\begin{proof}
Since $W$ is smooth, projective, geometrically integral over $K$, we have $\overline{\Qbb}[W]^{\times} = \overline{\Qbb}^{\times}$. According to our assumptions, by Proposition 3.1, we have $W(\Qbb) \not= \emptyset$, so $\textup{Br}_{0}\,W = \textup{Br}\,\Qbb$. Since $\Qbb$ is a number field, by the Hochschild--Serre spectral sequence, we have an isomorphism 
$$ \textup{Br}_{1}\,W/\textup{Br}_{0}\,W \simeq \textup{H}^{1}(\Qbb,\textup{Pic}\,\overline{W}). $$

By Proposition 4.5, the geometric Picard number of $W$ is equal to $8$ and a basis of $\textup{Pic}\,\overline{W}$ is given by 
$$ S = \{ D_{1}, D_{2}, D_{3}, C_{1}^{++}, C_{1}^{-+}, C_{2}^{++}, C_{2}^{-+}, C_{3}^{++} \} $$
or
$$ S' = \{ D_{1}, D_{2}, D_{3}, C_{1}^{++}, S_{1} := \frac{1}{2}(D_{1} + D_{2} + D_{3} + C_{1}^{++} + C_{1}^{-+}), C_{2}^{++}, S_{2} := \frac{1}{2}(D_{1} + D_{2} + D_{3} + C_{2}^{++} + C_{2}^{-+}), C_{3}^{++} \}$$
along with the corresponding intersection matrices. From properties of Wehler K3 surfaces (see Section 4.1) and the intersection pairings of the divisor classes in $\textup{Pic}\,\overline{W}$, we find that

\begin{equation}
\begin{cases}
D_{1} = C_{1}^{++} + C_{1}^{+-} = C_{1}^{-+} + C_{1}^{--}, \\
D_{2} = C_{2}^{++} + C_{2}^{+-} = C_{2}^{-+} + C_{2}^{--}, \\
D_{3} = C_{3}^{++} + C_{3}^{+-} = C_{3}^{-+} + C_{3}^{--}, \\
D_{1} + D_{2} + D_{3} = C_{1}^{++} + C_{1}^{-+} + C_{2}^{++} + C_{2}^{-+} + C_{3}^{++} + C_{3}^{-+}.
\end{cases}
\end{equation}

Consider the quadratic field extension $K = \Qbb(\sqrt{-km})$ of $\Qbb$, since $-km$ is not a square in $\Zbb$ as $-km \equiv -1$ (mod $7$). Now we study the action of the absolute Galois group on $\textup{Pic}\,\overline{W}$, which can be reduced to the action of $G = \textup{Gal}(K/\Qbb)$. One clearly has $G = \langle \sigma \rangle \cong \Zbb/2\Zbb$, where
$$ \sigma(\sqrt{-km}) = -\sqrt{-km}. $$
Given a finite cyclic group $G = \langle \sigma \rangle$ and a $G$-module $M$, by \cite[Proposition 1.7.1]{NSW15}, recall that we have isomorphisms $\textup{H}^{1}(G,M) \cong \hat{\textup{H}}^{-1}(G,M)$, where the latter group is the quotient of $\prescript{}{N_{G}}{M}$, the set of elements of $M$ of norm $0$, by its subgroup $(1 - \sigma)M$.

Note that for $1 \leq i \leq 3$, $\sigma(C_{i}^{\pm \pm}) = C_{i}^{\pm \mp} = D_{i} - C_{i}^{\pm \pm}$ in $\textup{Pic}\,\overline{W}$. Now we consider two cases as follows.

\textbf{Case 1:} If $S$ forms a basis of $\textup{Pic}\,\overline{W}$, then we have
$$ \textup{Ker}(1+\sigma) = \langle D_{1} - 2C_{1}^{++}, D_{2} - 2C_{2}^{++}, D_{3} - 2C_{3}^{++}, C_{1}^{++} - C_{1}^{-+}, C_{2}^{++} - C_{2}^{-+} \rangle, $$
and $\textup{Im}(1-\sigma) = \langle 2C_{1}^{++} - D_{1}, 2C_{1}^{-+} - D_{1}, 2C_{2}^{++} - D_{2}, 2C_{2}^{-+} - D_{2}, 2C_{3}^{++} - D_{3} \rangle = \langle D_{1} - 2C_{1}^{++}, D_{2} - 2C_{2}^{++}, D_{3} - 2C_{3}^{++}, 2(C_{1}^{++} - C_{1}^{-+}), 2(C_{2}^{++} - C_{2}^{-+} \rangle$.

By \cite[Proposition 1.6.7]{NSW15}, we have 
$$ \textup{H}^{1}(\Qbb,\textup{Pic}\,\overline{W}) = \textup{H}^{1}(G,\textup{Pic}\,\overline{W}), $$ where $G = \langle \sigma \rangle \cong \Zbb/2\Zbb$. Then $\textup{H}^{1}(G, \textup{Pic}\,\overline{W}) \cong \textup{Ker}(1+\sigma)/\textup{Im}(1-\sigma) \cong (\Zbb/2\Zbb)^{2}$.

\textbf{Case 2:} If $S'$ forms a basis of $\textup{Pic}\,\overline{W}$, then we have
$$ \textup{Ker}(1+\sigma) = \langle D_{1} - 2C_{1}^{++}, D_{2} - 2C_{2}^{++}, D_{3} - 2C_{3}^{++}, S_{1} - 2C_{1}^{++} - C_{2}^{++} - C_{3}^{++}, S_{2} - C_{1}^{++} - 2C_{2}^{++} - C_{3}^{++} \rangle, $$
and $\textup{Im}(1-\sigma) = \langle 2C_{1}^{++} - D_{1}, 2S_{1} - 2D_{1} - D_{2} - D_{3}, 2C_{2}^{++} - D_{2}, 2S_{2} - D_{1} - 2D_{2} - D_{3}, 2C_{3}^{++} - D_{3} \rangle$.

Then we also have $\textup{H}^{1}(G, \textup{Pic}\,\overline{W}) \cong \textup{Ker}(1+\sigma)/\textup{Im}(1-\sigma) \cong (\Zbb/2\Zbb)^{2}$.
\\~\\
\indent We keep the notation as above. Here for any divisor $D \in \textup{Pic}\,\overline{X}$, denote by $[D]$ its image in $\textup{Pic}\,\overline{U}$. Now by Proposition 4.1, $\textup{Pic}\,\overline{U}$ is given by the following quotient group
$$ \textup{Pic}\,\overline{U} \cong \textup{Pic}\,\overline{W}/(\Zbb D_{1} \oplus \Zbb D_{2} \oplus \Zbb D_{3}) = \langle [C_{1}^{++}], [C_{1}^{-+}], [C_{2}^{++}], [C_{2}^{-+}], [C_{3}^{++}] \rangle \hspace{0.5cm} (\textup{in Case 1}) $$
or
$$ \textup{Pic}\,\overline{U} \cong \textup{Pic}\,\overline{W}/(\Zbb D_{1} \oplus \Zbb D_{2} \oplus \Zbb D_{3}) = \langle [C_{1}^{++}], [S_{1}], [C_{2}^{++}], [S_{2}], [C_{3}^{++}] \rangle \hspace{0.5cm} (\textup{in Case 2}). $$
By Proposition 4.1, we also have $\overline{\Qbb}^{\times} = \overline{\Qbb}[U]^{\times}$. By the Hochschild--Serre spectral sequence, we have the following isomorphism 
$$ \textup{Br}_{1}\,U/\textup{Br}_{0}\,U \cong \textup{H}^{1}(\Qbb,\textup{Pic}\,\overline{U}) $$
as $\Qbb$ is a number field. Since $\textup{Pic}\,\overline{U}$ is free and $\textup{Gal}(\overline{\Qbb}/K)$ acts on $\textup{Pic}\,\overline{U}$ trivially, we obtain that $\textup{H}^{1}(\Qbb,\textup{Pic}\,\overline{U}) \cong \textup{H}^{1}(G,\textup{Pic}\,\overline{U})$. In both \textbf{Case 1} and \textbf{Case 2}, with the action of $G$, we can compute in the quotient group $\textup{Pic}\,\overline{U}$:
\begin{equation*}
\begin{cases}
\textup{Ker}(1+\sigma) = \textup{Pic}\,\overline{U},\\

\textup{Im}(1-\sigma) = \langle [2C_{1}^{++}], 2[C_{1}^{-+}], 2[C_{2}^{++}], 2[C_{2}^{-+}], 2[C_{3}^{++}] = 2\textup{Pic}\,\overline{U}, \rangle 
\end{cases}
\end{equation*}
so
$$ \textup{H}^{1}(\Qbb, \textup{Pic}\,\overline{U}) = \textup{H}^{1}(\langle \sigma \rangle, \textup{Pic}\,\overline{U}) = \frac{\textup{Ker}(1+\sigma)}{(1-\sigma)\textup{Pic}\,\overline{U}} = (\Zbb/2\Zbb)^{5}. $$

Now we produce some concrete generators in $\textup{Br}_{1}\,U$ for $\textup{Br}_{1}\,U/\textup{Br}_{0}\,U$. The affine scheme $U \subset \Abb^{3}$ is defined over $\Qbb$ by the equation
\begin{equation}
(x^{2} - 36)(y^{2} - 36)(z^{2} - 36) - m(xyz + C)^{2} - k = 0.
\end{equation}
Here we note that 
$$ \{x^{2} - 36 = 0\} \cap \{(y^{2} - 36)(z^{2} - 36) = 0\}, $$
$$ \{y^{2} - 36 = 0\} \cap \{(x^{2} - 36)(z^{2} - 36) = 0\}, $$
$$ \{z^{2} - 36 = 0\} \cap \{(x^{2} - 36)(y^{2} - 36) = 0\} $$
are closed subsets of codimension $\geq 2$ on $U$. By Grothendieck's purity theorem (\cite[Theorem 6.8.3]{Poo17}), we have the exact sequence
$$ 0 \rightarrow \textup{Br}\,U \rightarrow \textup{Br}\,\Qbb(U) \rightarrow \bigoplus_{D \in U^{(1)}} \textup{H}^{1}(\Qbb(D),\Qbb/\Zbb), $$
where the last map is given by the residue along the codimension-one point $D$. We consider the following quaternion algebras:
$$ \mathcal{A} = (x^{2} - 36, -km), $$
$$ \mathcal{B} = (y^{2} - 36, -km). $$
In order to prove that $\mathcal{A},\mathcal{B}$ come from a class in $\textup{Br}\,U$, it suffices to show that their residues along the irreducible components of the divisors that belong to $\{x^{2} - 36 = 0\}$ and $\{y^{2} - 36 = 0\}$ are all trivial. Indeed, in the residue field of each one of these irreducible divisors, $-km$ is clearly a square; standard formulae for residues in terms of the tame symbol \cite[Example 7.1.5, Proposition 7.5.1]{GS17} therefore show that $\mathcal{A},\mathcal{B}$ are unramified, so they are elements of $\textup{Br}\,U$ and moreover they are clearly algebraic. We also have $$(x^{2} - 36, -km) + (y^{2} - 36, -km) + (z^{2} - 36, -km) = 0$$ in $\textup{Br}_{1}\,U$. The residues of $\mathcal{A},\mathcal{B}$ at the irreducible divisors $D_{1}, D_{2}, D_{3}$ given by $rst = 0$ which form the complement of $U$ in $W$ are easily seen to be trivial. One thus has $\mathcal{A},\mathcal{B} \in \textup{Br}_{1}\,W$.

It is also easy to see that $\mathcal{A}_{1} = (x - 6, -km)$, $\mathcal{A}_{2} = (x + 6, -km)$, $\mathcal{B}_{1} = (y - 6, -km)$, $\mathcal{B}_{2} = (y + 6, -km)$, $\mathcal{C}_{1} = (z - 6, -km)$, and $\mathcal{C}_{2} = (z + 6, -km)$ are elements of $\textup{Br}_{1}\,U$ but not of $\textup{Br}_{1}\,W$. These 6 elements satisfy that $\mathcal{A} = \mathcal{A}_{1} + \mathcal{A}_{2}$ and similarly for $\mathcal{B}, \mathcal{C}$; and $\mathcal{A}_{1} + \mathcal{A}_{2} + \mathcal{B}_{1} + \mathcal{B}_{2} + \mathcal{C}_{1} + \mathcal{C}_{2} = 0$ in $\textup{Br}_{1}\,U$.
%The elements $\mathcal{A}_{i}$ ($i = 1,2$) are non-constant by the Faddeev exact sequence (see \cite[Theorem 1.5.2]{CTS21}), since the pull-backs of $\mathcal{A}_{i}$ from $\textup{Br}\,\Qbb(\Pbb^{1}) = \textup{Br}\,\Qbb(t)$ (via $\pi_{i} : W \rightarrow \Pbb^{1}$) are not constant as they give nontrivial residues at the closed point $(4t^{2}+1)$ of $\Pbb^{1}_{\Qbb}$. Meanwhile, the residue of $\mathcal{B}$ along the curve $C^{0} \subset \Abb^{2}$ given by $16x^{2}y^{2}-4x^{2}-4y^{2}-1 = 0$ is the class of the restriction $\left( (4k-5)^{2}-32 \right)|_{C^{0}}$, which is nontrivial, so $\mathcal{B}$ is also non-constant.
The elements $\mathcal{A}_{i}, \mathcal{B}_{i}, \mathcal{C}_{i}$ ($i \in \{1,2\}$) will contribute to the Brauer--Manin obstruction to the integral Hasse principle for $\mathcal{U}$ (the integral model of $U$ defined over $\Zbb$ by the same affine equation) in the next section. %(are linearly independent modulo $2$ in $\textup{Br}_{1}\,U/\textup{Br}_{0}\,U$)

In conclusion, we have $\textup{Br}_{1}\,W/\textup{Br}_{0}\,W \cong (\Zbb/2\Zbb)^{2}$, which can be seen as a subgroup of $\textup{Br}_{1}\,U/\textup{Br}_{0}\,U \cong (\Zbb/2\Zbb)^{5}$ on using the vanishing of $\textup{H}^{1}(G, \oplus_{i=1}^{3} \Zbb D_{i})$ (and we also have $\textup{Br}_{0}\,W = \textup{Br}_{0}\,U = \textup{Br}\,\Qbb$ since the natural composite map $\textup{Br}\,\Qbb \hookrightarrow \textup{Br}_{1}\,W \hookrightarrow \textup{Br}_{1}\,U$ is injective).
\end{proof}

\begin{remark}
If we compare the results in \cite{Dao24b}, we can see that the explicit elements of the algebraic Brauer group in \cite[Theorem 3.7]{Dao24b} are more complicated than those in Theorem 4.5 above, partly because the Diophantine equation that we consider here contains a part of the form $(x^{2} - A^{2})(y^{2} - A^{2})(z^{2} - A^{2})$ with $A \in \Zbb$. Furthermore, as also mentioned in \cite{Dao24b}, it would be interesting if one can compute the transcendental part of the Brauer group for this family of Markoff-type K3 surfaces like what the authors in \cite{LM20} and \cite{CTWX20} did for Markoff surfaces, which in general should be difficult.
\end{remark}

\section{Brauer--Manin obstruction from quaternion algebras}
Now we consider an explicit family of MK3 surfaces over $\Qbb$. We consider the surfaces in Proposition 3.1 with $$m = m_{0} = 3\ell(1 - \ell) = -36.13 = -468 \;(\ell = -12), C = C_{0} = -4330,$$ and $B = k \in \Zbb$. From now on, we denote by $W_{k}$ the projective MK3 surfaces, $U_{k}$ the affine open subvariety defined by $W_{k} \setminus \{rst = 0\}$ and $\mathcal{U}_{k}$ the integral model of $U_{k}$ defined by the same equation.

\subsection{Existence of local integral points}
First of all, we study the existence of local integral points on the affine MK3 surfaces. Note that now we consider $m_{0} = -468$ and $C_{0} = -4330$.

\begin{proposition}
For $k \in \Zbb$, let $W_{k} \subset \Pbb^{1} \times \Pbb^{1} \times \Pbb^{1}$ be the MK3 surface defined over $\Qbb$ by the $(2,2,2)$-form 
\begin{equation}
    F(x,y,z) = (x^{2} - 36)(y^{2} - 36)(z^{2} - 36) - m_{0}(xyz + C_{0})^{2} - k = 0.
\end{equation}
Let $\mathcal{U}_{k}$ be the integral model of $U_{k}$ defined over $\Zbb$ by the same equation. If $k$ satisfies the condition:
\begin{enumerate}
    \item[$\star$] $k > 0$ and $k \equiv 1$ (\textup{mod} $P$), where $P = 2^{3} \times 3 \times 5 \times 7 \times 11 \times 13 \times 31 \times 433 \times 2017 \times 3253 \times 8501 \times 32687 \times 46649 \times 4057231$,
\end{enumerate}
then $\mathcal{U}_{k}(\textbf{\textup{A}}_{\Zbb}) \not= \emptyset$.
\end{proposition}

\begin{proof}
For the place at infinity, it is clear that there exists a real point on $U_{k}$:
\begin{enumerate}
	\item[(a)] If $k = 1$, then we find a real solution $(x,y,z)$ such that $xyz = 4330$ and $y = z$. We obtain the following system of equations: $xy^{2} = 4330$ and $(x^{2} - 36)(y^{2} - 36)^{2} = 1$. Equivalently, we have $x = 4330/y^{2}$ ($y \neq 0$) and $36y^{4}(y^{2} - 36)^{2} + y^{4} - 4330(y^{2} - 36)^{2} = 0$. The latter equation has a real solution $y \in (0, +\infty)$ by the intermediate value theorem, so we obtain a real point $(x,y,z) = (x,y,y)$ on $U$.
	
	\item[(b)] If $k > 1$ then $k > 2^3 \times 3 \times 5 \times 11 \times 13 \times 31 \time 433 \times 2017 = 1084050607640 > 8774438544 = 13 \times 36 \times 4330^2 - 36^3$, so we have $(x,y,z) = (\sqrt{k - 8774438544},0,0) \in U_{k}(\Rbb)$.
\end{enumerate}

For the existence of local integral points at finite places $p$ on $\mathcal{U}_{k}$, we have:
\begin{enumerate}
	\item[(i)] Prime powers of $p = 2$: It is clear that every solution modulo $2$ is singular. Thanks to the condition $\star$, we find the solution $(1,1,1)$ mod $8$ (with $F'_{x} \equiv 2$ mod $8$), which then lifts to a $2$-adic integer solution by Hensel's lemma.
	
	\item[(ii)] Prime powers of $p \in \{3, 5, 7, 13\}$: Thanks to the condition $\star$, we find the non-singular solution $(1,1,1)$ for $p = 3$, $(0,2,2)$ for $p = 5$, $(2,3,2)$ for $ p = 7$, and $(0,1,3)$ for $p = 13$, which then respectively lift to a $3$-adic and $5$-adic integer solution by Hensel's lemma (with respect to $x$, fixing $y,z$).
	
	\item[(iii)] Prime powers of $p \geq 11$ ($p \neq 13$): For simplicity, we fix $z = 0$ over $\Zpbb$. The equation becomes
	$$ F(x,y,0) = -36(x^{2}-36)(y^{2}-36) - m_{0}.C_{0}^{2} - k = 0 $$
	which defines an affine curve $C \subset \Abb^{2}_{(x,y)}$ over $\Fpbb$. First we consider its projective closure in $\Pbb^{2}_{[x:y:t]}$ defined by
	$$ -36(x^{2}-36t^{2})(y^{2}-36t^{2}) - (m_{0}C_{0}^{2} + k)t^{4} = 0. $$
	If $p$ does not divide either $m_{0}C_{0}^{2} + k$ or $m_{0}C_{0}^{2} + k + 36^{3}$, then the projective curve has only two singularities which are \emph{ordinary} of multiplicity $2$, namely $[1:0:0]$ and $[0:1:0]$. By the genus--degree formula and the fact that the geometric genus is a birational invariant, we obtain
	$$ g(C) = \frac{(\deg C - 1)(\deg C - 2)}{2} - \sum_{i = 1}^{n} \frac{r_{i}(r_{i}-1)}{2}, $$
	where $n$ is the number of ordinary singularities and $r_{i}$ is the multiplicity of each singularity for $i = 1,\dots,n$; in particular, $g(C) = 3 - 2 = 1$.
	
	Now we consider the original projective closure $C^{1} \subset \Pbb^{1}_{[x:r]} \times \Pbb^{1}_{[y:s]}$ defined by
	$$ -36(x^{2}-36r^{2})(y^{2}-36s^{2}) - (m_{0}C_{0}^{2} + k)r^{2}s^{2} = 0. $$
	If $p$ does not divide either $m_{0}C_{0}^{2} + k$ or $m_{0}C_{0}^{2} + k + 36^{3}$, then the projective curve $C^{1}$ is \textbf{smooth} over $\Fpbb$ under our assumption on $k$. Then by the Hasse--Weil bound for smooth, projective and geometrically integral curves of genus $1$, we have
	$$ |C^{1}(\Fpbb)| \geq p + 1 - 2\sqrt{p} = (\sqrt{p} - 1)^{2}, $$
	so $|C^{1}(\Fpbb)| \geq 5$ as $p \geq 11$. As $C^{1}$ has exactly $4$ points at infinity (when $rs = 0$), the affine curve $C$ has at least one \emph{smooth} $\Fpbb$-point which then lifts to a $p$-adic integral point by Hensel's lemma.
	\\~\\
	\indent Next, consider the case when $p$ divides $m_{0}C_{0}^{2} + k$ or $m_{0}C_{0}^{2} + k + 36^{3}$. Since $433$ divides $C_{0}$ and $k \equiv 1$ mod $433$, we have $p \neq 433$. We will fix instead $z=1$ over $\Zpbb$. The equation becomes
	$$ F(x,y,1) = -35(x^{2}-36)(y^{2}-36) - k - m_{0}(xy+C_{0})^{2} = 0 $$
	which defines an affine curve $D \subset \Abb^{2}_{(x,y)}$ over $\Fpbb$. If we consider its projective closure in $\Pbb^{2}_{[x:y:t]}$ defined by
	$$ -35(x^{2}-36t^{2})(y^{2}-36t^{2}) - kt^{4} - m_{0}(xy+C_{0}t^{2})^{2} = 0, $$
	then the projective curve also has only two singularities which are \emph{ordinary} of multiplicity $2$, namely $[1:0:0]$ and $[0:1:0]$. By the genus--degree formula and the fact that the geometric genus is a birational invariant, we obtain
	$$ g(D) = \frac{(\deg D - 1)(\deg D - 2)}{2} - \sum_{i = 1}^{n} \frac{r_{i}(r_{i}-1)}{2}, $$
	where $n$ is the number of ordinary singularities and $r_{i}$ is the multiplicity of each singularity for $i = 1,\dots,n$; in particular, $g(D) = 3 - 2 = 1$.
	
	Now we consider the original projective closure $D^{1} \in \Pbb^{1}_{[x:r]} \times \Pbb^{1}_{[y:s]}$ defined by
	$$ -35(x^{2}-36r^{2})(y^{2}-36s^{2}) - kr^{2}s^{2} - m_{0}(xy+C_{0}rs)^{2} = 0. $$
	The projective curve $D^{1}$ is \textbf{smooth} over $\Fpbb$ under the condition $\star$. Then by the Hasse--Weil bound for smooth, projective and geometrically integral curves of genus $1$, we have $|D^{1}(\Fpbb)| \geq 5$ since $p \geq 11$. As $D^{1}$ has at most $4$ points at infinity (when $rs = 0$), the affine curve $D$ has at least one \emph{smooth} $\Fpbb$-point which then lifts to a $p$-adic integral point by Hensel's lemma. Our proof is now complete.
\end{enumerate}
\end{proof}

\begin{remark}
Under the assumptions of Proposition 5.1, we know from Proposition 3.2 that there exists a Zariski-dense set of $\Qbb$-points at infinity (when $rst = 0$) on these MK3 surfaces.
\end{remark}

\subsection{Integral Brauer--Manin obstruction}
It is important to recall that there always exist $\Qbb$-points (at infinity) on every member $W_{k} \subset \Pbb^{1} \times \Pbb^{1} \times \Pbb^{1}$ of the above family of Markoff-type K3 surfaces, hence they satisfy the (rational) Hasse principle. Now we prove the Brauer--Manin obstruction to the integral Hasse principle for the integral model $\mathcal{U}_{k}$ of the affine subvariety $U_{k} \subset W_{k}$ by calculating the local invariants for some quaternion algebra classes $\mathcal{A}$ in the Brauer group of $U_{k}$:
$$ \textup{inv}_{p}\,\mathcal{A} : \mathcal{U}_{k}(\Zpbb) \rightarrow \Zbb/2\Zbb, \hspace{1cm} \textbf{u} = (x,y,z) \mapsto \textup{inv}_{p}\,\mathcal{A}(\textbf{u}). $$

\begin{theorem}
For $k \in \Zbb$, let $W_{k} \subset \Pbb^{1} \times \Pbb^{1} \times \Pbb^{1}$ be the MK3 surface defined over $\Qbb$ by the $(2,2,2)$-form 
\begin{equation}
    F(x,y,z) = (x^{2} - 36)(y^{2} - 36)(z^{2} - 36) - m_{0}(xyz + C_{0})^{2} - k = 0.
\end{equation}
Let $\mathcal{U}_{k}$ be the integral model of $U_{k}$ defined over $\Zbb$ by the same equation. If $k$ satisfies the condition:
\begin{enumerate}
    \item $k = \ell^{2}$ where $\ell \in \Zbb$ and $\ell > 0$ such that $\ell \equiv 1$ (\textup{mod} $P$), where $P = 2^{3} \times 3 \times 5 \times 7 \times 11 \times 13 \times 31 \times 433 \times 2017 \times 3253 \times 8501 \times 32687 \times 46649 \times 4057231$;
    \item $(p,13)_{p} = 0$ for any prime divisor $p$ of $\ell$,
\end{enumerate}
then $\mathcal{U}_{k}(\textbf{\textup{A}}_{\Zbb}) \not= \emptyset$ and there is an algebraic Brauer--Manin obstruction to the integral Hasse principle for $\mathcal{U}_{k}$, i.e., $\mathcal{U}_{k}(\Zbb) \subset \mathcal{U}_{k}(\textbf{\textup{A}}_{\Zbb})^{\textup{Br}_{1}} = \emptyset$.
\end{theorem}

\begin{proof}
Since $k$ satisfies $\star$, we have the result on the algebraic Brauer group of $U_{k}$ by Theorem 4.7 and the existence of local integral points on $\mathcal{U}_{k}$ by Proposition 5.1. For any local point in $\mathcal{U}_{k}(\textbf{A}_{\Zbb})$, we calculate its local invariants at every prime $p \leq \infty$. First of all, note that $U_{k}$ is smooth over $\Qbb$ and the affine equation 
$$ (x^{2} - 36)(y^{2} - 36)(z^{2} - 36) = \ell^{2} - 13.36(xyz - 4330)^{2} $$ 
implies $$(x^{2} - 36, 13) + (x^{2} - 36, 13) + (x^{2} - 36, 13) = 0$$ in $\textup{Br}_{1}\,U_{k}/\textup{Br}_{0}\,U_{k}$. Now by abuse of notation, at each place $p$ of $\Qbb$, we consider a local point denoted by $(x,y,z)$.

At $p = \infty$: Since $13 > 0$, for any $(x,y,z) \in U(\Rbb)$, we have $\textup{inv}_{\infty}\,(x \pm 6, 13)_{\infty} = 0$.

At $p = 2$: Since $k \equiv 1$ mod $8$ and $m_{0} \equiv 4$ mod $8$, all of the coordinates $x,y,z$ are in $\Zbb_{2}^{\times}$. Furthermore, $13 \equiv 1$ mod $4$, so $(x \pm 6, 13)_{2} = 0$.

At $p = 3$: Since $k \equiv 1$ mod $3$, all of the coordinates $x,y,z$ are in $\Zbb_{3}^{\times}$, so $(x \pm 6, 13)_{3} = 0$.

At $p = 13$: Since $(x^{2}-36)(y^{2}-36)(z^{2}-36) \equiv 1$ mod $13$, computation modulo $13$ shows that there must exist at least one coordinate, without loss of generation (WLOG) let it be $x$, such that $(x - 6, 13)_{13} = 1/2$ or $(x + 6, 13)_{13} = 1/2$.

At $p \geq 5$ and $p \neq 13$: Since $k = \ell^{2}$ and every odd prime divisor $p$ of $\ell$ satisfies $(13,p)_{p} = 0$, we find that $(x \pm 6, 13)_{p} = 0$.

In conclusion, we have 
$$ \sum_{p \leq \infty} (x - 6, 13)_{p} = \frac{1}{2} \not= 0 \hspace{0.5cm}\textup{or}\hspace{0.5cm} \sum_{p \leq \infty} (x + 6, 13)_{p} = \frac{1}{2} \not= 0, $$
so $\mathcal{U}_{k}(\Zbb) \subset \mathcal{U}_{k}(\textbf{\textup{A}}_{\Zbb})^{\textup{Br}_{1}} = \emptyset$.
\end{proof}

\begin{example}
By taking $k = 1$, we have a family of MK3 surfaces that contains a Zariski-dense set of $\Qbb$-points and do not contain any $\Zbb$-point due to the Brauer--Manin obstruction. The affine equation of these MK3 surfaces is as follows.
$$ (x^{2} - 36)(y^{2} - 36)(z^{2} - 36) + 13.36(xyz - 4330)^{2} - 1 = 0. $$
\end{example}

\subsection{Counting the Hasse failures}
Recall that in the case of Markoff surfaces, we see from \cite{LM20} that the number of counterexamples to the integral Hasse principle which can be explained by the Brauer--Manin obstruction is asymptotically equal to $M^{1/2}/(\log M)^{1/2}$; by \cite{CTWX20}, this number is also the asymptotic lower bound for the number of Markoff surfaces such that there is no Brauer--Manin obstruction to the integral Hasse principle, which is slightly better than the result $M^{1/2}/\log M$ in \cite{LM20}.

In this part, we study the number of examples where local integral points exist as well as the number of counterexamples to the integral Hasse principle for our Markoff-type K3 surfaces which can be explained by the Brauer--Manin obstruction. More precisely, we calculate the natural density of $k \in \Zbb$ satisfying the hypotheses of Proposition 5.1 and Theorem 5.2.

\begin{theorem}
For the above family of MK3 surfaces, we have
	$$ \# \{k \in \Zbb: |k| \leq M,\ \mathcal{U}_{k}(\textbf{\textup{A}}_{\Zbb}) \not= \emptyset \} \asymp M $$
	and
	$$ \# \{k \in \Zbb: |k| \leq M,\ \mathcal{U}_{k}(\textbf{\textup{A}}_{\Zbb}) \not= \emptyset,\ \mathcal{U}_{k}(\textbf{\textup{A}}_{\Zbb})^{\textup{Br}} = \emptyset \} \gg \frac{M^{1/2}}{\textup{log}\,M}, $$
as $M \rightarrow +\infty$.
\end{theorem}

\begin{proof}
This proof is similar to that of \cite[Theorem 4.8]{Dao24b}.

For the first estimate, the result follows directly from the fact that the hypothesis of Proposition 5.1 only gives finitely many congruence conditions on $k$, so the total numbers of $k$ are approximately a proportion of $M$ as $M \rightarrow +\infty$.

For the second estimate, following the hypothesis of Theorem 5.2, we only give an asymptotic lower bound by considering the case when $\ell$ is a \textbf{prime}. The result follows from the fact that if $k$ satisfies the hypothesis of Theorem 5.2, then $|k| = \ell^{2} \leq M$ and as $M \rightarrow +\infty$, the number of primes $\ell < \sqrt{M}$ that satisfy finitely many congruence conditions is asymptotically equal to $\ds\frac{\sqrt{M}}{\log M}$ (see \cite[Section 7.9]{Apo76}).
\end{proof}

\begin{remark}
The second counting result can be improved by applying Lemma 5.19 in \cite{LM20}. We expect that the lower bound can be improved further with some suitable changes to the conditions on $k$ in Proposition 5.1 and Theorem 5.2. The most significant improvement should be about the power of $M$ on the numerator in the counting result if possible. The fact that the current power of $M$ equals $1/2$ is related to the choice of $k$ to be a \textbf{quadratic} function of $\ell$ and the counting is implemented on $\ell$. Furthermore, the \emph{algebraic} Brauer group modulo the constant Brauer group of the MK3 surfaces $U_{k}$ in question is isomorphic to $(\Zbb/2\Zbb)^{5}$ and generated by classes of \emph{quartenion algebras}, so working with quadratic functions such as squares is more convenient.

It would be interesting if we can compute the \emph{transcendental} Brauer group of $U_{k}$, which would help us consider the Brauer--Manin set with respect to the whole Brauer group instead of only its algebraic part. If we can at least find some explicit elements of \emph{odd} order (if they exist) in the transcendental Brauer group, then we may obtain some different power of $M$ that contributes to the counting result.
\end{remark}

\section{Further remarks}
In this section, we consider the integral Hasse principle and strong approximation for general MK3 surfaces as introduced in Section 2.2. The equation of MK3 surfaces that we study in this paper is more general in some sense than that of MK3 surfaces in \cite{Dao24b}, but it is not the most general since it depends on only \emph{one} parameter.

\subsection{Brauer groups and Picard groups}
We prove that it is possible to find some explicit elements in the algebraic Brauer group for a family of MK3 surfaces with more parameters.

Indeed, for $a, b, c, d, e \in \Zbb$, let $W \subset \Pbb^{1} \times \Pbb^{1} \times \Pbb^{1}$ be a \emph{smooth} MK3 surface defined over $\Qbb$ by the $(2,2,2)$-form 
$$ F(x, y, z) = ax^{2}y^{2}z^{2} + b(x^{2}y^{2} + x^{2}z^{2} + y^{2}z^{2}) + cxyz + d(x^{2} + y^{2} + z^{2}) + e = 0. $$
Let $\mathcal{U}$ be the integral model of $U$ defined over $\Zbb$ by the same equation. Assume further that $a, b, c, d, e$ are all \textbf{nonzero} and $ad \neq b^{2}$ so that we can rewrite the equation as follows. Note that if $W$ is a \emph{non-degenerate} MK3 surface, then we have $ad - b^{2} \neq 0, d^{2} - be \neq 0$, and $c \neq 0$ (see Section 2.2).

\begin{equation}
\begin{aligned}
&(abd - b^{3})x^{2}y^{2}z^{2} + (bx^{2} + d)(by^{2} + d)(bz^{2} + d) + bcdxyz + bde - d^{3} = 0\\
\Leftrightarrow &(bx^{2} + d)(by^{2} + d)(bz^{2} + d) + (abd - b^{3})\left(x^{2}y^{2}z^{2} + \frac{bcd}{abd - b^{3}}xyz\right) = d^{3} - bde\\ 
\Leftrightarrow &(bx^{2} + d)(by^{2} + d)(bz^{2} + d) + (abd - b^{3})\left(xyz + \frac{bcd}{2(abd - b^{3})}\right)^{2} = d^{3} - bde - \frac{b^{2}c^{2}d^{2}}{4(abd - b^{3})}\\
\Leftrightarrow &4(abd - b^{3})(bx^{2} + d)(by^{2} + d)(bz^{2} + d) + (2(abd - b^{3})xyz + bcd)^{2} = 4(abd - b^{3})(d^{3} - bde) - b^{2}c^{2}d^{2}\\
\Leftrightarrow &A(bx^{2} + d)(by^{2} + d)(bz^{2} + d) = k - (2(abd - b^{3})xyz + bcd)^{2},
\end{aligned}
\end{equation}
where $A = 4(abd - b^{3}) = 4b(ad - b^{2})$ and $k = 4(abd - d^{3})(d^{3} - bde) - b^{2}c^{2}d^{2} = 4bd(ad - b^{2})(d^{2} - be) - b^{2}c^{2}d^{2}$.

Following the same method in the paper, we can find some explicit elements in the algebraic Brauer group of $W$ given by the classes of $(bx^{2} + d, k)$, $(by^{2} + d, k)$, and $(bz^{2} + d, k)$. We can use them to study the Brauer--Manin obstruction for integral points on $\mathcal{U}$. Another natural question is how to compute the whole algebraic Brauer groups of $W$ and $U$. For that, we need to compute the geometric Picard groups of $W$ and $U$ if we follow the same method as above.

By Proposition 4.3, if $a \equiv 0, b \equiv -1, c \equiv d \equiv 1, e \equiv 4$ (mod 7) then the affine equation of $W$ is the same as the equation considered in this proposition after taking congruence modulo $7$. Therefore, the Picard numbers of $\overline{W}$ and $\overline{U}$ are equal to $8$ and $5$, respectively. Following the same method, we can try to compute the algebraic Brauer groups of $W$ and $U$ as well as study the Brauer--Manin obstruction to the integral Hasse principle for $\mathcal{U}$.

\subsection{Failure of strong approximation}
We keep the notation as in the previous part. Assume that $a, b, c, d, e$ are all nonzero and that $W$ is non-degenerate. We also assume that $a \equiv 0, b \equiv -1, c \equiv d \equiv 1, e \equiv 4$ (mod 7) (so we have the above result about the Picard numbers of $W$ and $U$). For the family of MK3 surfaces $W$ defined by $F = 0$, from Section 3 we can see that the divisor $K + D$ is \emph{big}, where $K$ is the (trivial) canonical divisor on $W$ and $D := D_{1} + D_{2} + D_{3}$ is an \emph{ample} divisor. Therefore, $U = W \setminus D$ is of \emph{log general type}, and Vojta's Conjecture asserts that integral points on $\mathcal{U}$ are \emph{not Zariski-dense} (see \cite[Section 1.2]{Cor16}).

By the proof of \cite[Lemma 1.2]{Rap13}, if the affine $\Qbb$-variety $U$ has strong approximation off $\{\infty\}$, $\mathcal{U}(\Zbb)$ must be Zariski-dense in $U$. We now prove that integral points on $\mathcal{U}$ are not Zariski-dense as predicted by Vojta's Conjecture. As a result, strong approximation off $\{\infty\}$ fails on $\mathcal{U}$, and so does strong approximation.

\begin{proposition}
$\mathcal{U}(\Zbb)$ is not Zariski-dense in $U$.
\end{proposition}

\begin{proof}
We prove that $\mathcal{U}(\Zbb)$ is finite, and so it is not Zariski-dense in $U$. Indeed, from Equation (2), we have: 
$$ |ax^{2}y^{2}z^{2} + b(x^{2}y^{2} + x^{2}z^{2} + y^{2}z^{2}) + cxyz + d(x^{2} + y^{2} + z^{2})| = |e|. $$
By the triangular inequality, $|a||x^{2}y^{2}z^{2}| - |b|(x^{2}y^{2} + x^{2}z^{2} + y^{2}z^{2}) - |c||xyz| - |d|(x^{2} + y^{2} + z^{2}) \leq |e|$. Recall that we consider that $a, b, c, d, e$ are all \emph{nonzero}.

If $xyz = 0$, without loss of generality (WLOG) assume that $z = 0$, from $(11)$ we obtain
$$ Ad(bx^{2} + d)(by^{2} + d) = k - b^{2}c^{2}d^{2}. $$ Note that $bd \equiv -1$ (mod 7) and $bd \neq 0$, so $bx^{2} + d$ and $by^{2} + d$ are nonzero for $x, y \in \Zbb$. We also have $Ad \neq 0$ and $k - b^{2}c^{2}d^{2} \neq 0$. Therefore, there are only finitely many solutions to this equation for $x, y \in \Zbb$.

If $|xyz|$ is nonzero and bounded, then each of $|x|, |y|, |z|$ is also bounded, hence the finiteness of $\mathcal{U}(\Zbb)$.

If $|xyz|$ is nonzero and unbounded, assume that we have integer solutions such that $|xyz| \geq \max\left\{\ds\frac{12|b|}{|a|}, \ds\frac{4|c|}{|a|}, \ds\frac{12|d|}{|a|}\right\}$. Note that if $|xyz| \neq 0$, then $x^{2} \geq 1$, $y^{2} \geq 1$, and $z^{2} \geq 1$ for $x, y, z \in \Zbb$. As a result, we obtain $\ds\frac{|a|}{4}x^{2}y^{2}z^{2} \geq |b|(x^{2}y^{2} + x^{2}z^{2} + y^{2}z^{2})$, $\ds\frac{|a|}{4}x^{2}y^{2}z^{2} \geq |c||xyz|$, and $\ds\frac{|a|}{4} \geq |d|(x^{2} + y^{2} + z^{2})$. Therefore, $\ds\frac{|a|}{4}x^{2}y^{2}z^{2} \leq |e|$, i.e., $|xyz| \leq 2\sqrt{\ds\frac{|e|}{|a|}}$. Then $|xyz|$ is bounded, which is a contradiction.

In conclusion, $\mathcal{U}(\Zbb)$ is finite and so it is not Zariski-dense in $U$.
\end{proof}

\printbibliography

\textsc{University of Science and Technology of Hanoi, A21 Building, Vietnam Academy of Science and Technology, 18 Hoang Quoc Viet, Cau Giay, Hanoi, Vietnam}\\
\textit{E-mail address}: \href{mailto:dao-quang.duc@usth.edu.vn}{\texttt{dao-quang.duc@usth.edu.vn}}

\end{document}